\documentclass[a4paper,11pt]{article}
\usepackage{amsmath,amsthm}
\usepackage{thm-restate}
\usepackage{authblk}
\usepackage{aligned-overset}
\usepackage[square,sort,comma,numbers]{natbib}
\usepackage{booktabs}
\usepackage{cancel}
\usepackage[font=small]{caption}
\usepackage[hmargin={26mm,26mm},vmargin={30mm,35mm}]{geometry}
\usepackage{nicefrac}
\usepackage{paralist}
\usepackage[colorlinks, allcolors=blue]{hyperref}
\usepackage{subcaption}
\usepackage[compact,small]{titlesec}
\usepackage{tikz}
\usepackage{tikz-cd}
\usepackage{pgfplotstable}
\usepackage{newtxmath}
\usepackage{newtxtext}
\usepackage{enumitem}
\usepackage{bm}
\usepackage{algpseudocode}
\usepackage{algorithm}
\usepackage{thmtools}
\usepackage{thm-restate}
\usepackage{xparse}
\usepackage{empheq}

\usepackage{graphicx}
\graphicspath{{./pics/}}

\newcommand{\email}[1]{\href{mailto:#1}{#1}}

\allowdisplaybreaks

\usepackage[normalem]{ulem}
\newcounter{corr}
\definecolor{violet}{rgb}{0.580,0.,0.827}
\newcommand{\corr}[3]{\typeout{Warning : a correction remains in page \thepage}
  \stepcounter{corr}
	      {\color{blue}\ifmmode\text{\,\sout{\ensuremath{#1}}\,}\else\sout{#1}\fi}
              {\color{red}#2}
              {\color{violet} #3}
}

\newtheorem{theorem}{Theorem}
\newtheorem{proposition}[theorem]{Proposition}
\newtheorem{lemma}[theorem]{Lemma}
\newtheorem{corollary}[theorem]{Corollary}
\theoremstyle{remark}
\newtheorem{remark}[theorem]{Remark}
\theoremstyle{definition}

\newtheorem{definition}[theorem]{Definition}

\newtheorem*{theorem*}{Theorem}

\newtheorem{openquestion}{Open question}

\newcommand{\Real}{\mathbb{R}}
\DeclareMathOperator{\tr}{t}
\DeclareMathOperator{\ntr}{n}
\newcommand{\mc}{\mathcal}

\newcommand{\ext}{\mathrm{ext}}
\newcommand{\diam}{\mathrm{diam}}
\newcommand{\dist}{\mathrm{dist}}

\newcommand{\cext}[1]{E^{#1}}

\newcommand{\EHLZ}{E_\mathrm{HLZ}}
\newcommand{\CHLZ}{C_\mathrm{HLZ}}

\begin{document}

\title{Uniformly Bounded Cochain Extensions and Uniform Poincaré Inequalities}

\author[1]{Erik Nilsson}
\author[1]{Silvano Pitassi}
\affil[1]{IMAG, Univ Montpellier, CNRS, Montpellier, France\\
\email{erik.nilsson@umontpellier.fr}, \email{silvano.pitassi@umontpellier.fr}%
}

\maketitle

\begin{abstract} 
In this paper, we construct a novel global bounded cochain extension operator for differential forms on Lipschitz domains. Building upon the classical universal extension of Hiptmair, Li, and Zou, our construction restores global commutativity with the exterior derivative in the natural $H\Lambda^k(\Omega)$ setting. The construction applies to domains and ambient extension sets of arbitrary topology, with strict commutation holding on the orthogonal complement of harmonic forms, as dictated by the underlying topological obstruction. This provides a missing analytical tool for the rigorous foundation of Cut Finite Element Methods (CutFEM).
We also obtain continuous uniform Poincaré inequalities and lower bounds for the first Neumann eigenvalue on non-convex domains.
\end{abstract}

%------------------------------------------------------------------------------%
\section{Introduction}
The problem of extending Sobolev functions and vector fields from a bounded Lipschitz domain $\Omega$ to the whole space $\Real^n$ is classical in analysis. In the context of differential forms, Hiptmair, Li, and Zou \cite{hiptmair2012universalextension} constructed a graded family $\{\EHLZ^k\}_{k=0}^n$ of extension operators, which we refer to as the \emph{HLZ-extension} hereafter. This family is defined for all form degrees $k$ and is inspired by the classical approach of Stein \cite{stein1970singular}. 
Their main result can be summarised as follows, where the definition of the space $H^{(s,s)}\Lambda^k(\Omega)$ is given by \eqref{def:Hmm.space}.
\begin{theorem*}[Hiptmair, Li, Zou \cite{hiptmair2012universalextension}]
Let $\Omega \subset \Real^n$ be a bounded Lipschitz domain. For any form degree $k$ and any regularity parameter $s \ge 0$, there exists an extension operator
$\EHLZ^k : H^{(s,s)}\Lambda^k(\Omega) \to H^{(s,s)}\Lambda^k(\Real^n)$ and a constant $\CHLZ$, which depends on the domain $\Omega$ only in terms of its Lipschitz character, such that:
\begin{enumerate}[label=(\roman*)]
\item $\EHLZ^k\omega|_{\Omega} = \omega$ a.e.\ in $\Omega$.
\item $\|\EHLZ^k\omega\|_{H^{s}\Lambda^k(\Real^n)} \le \CHLZ \|\omega\|_{H^{s}\Lambda^k(\Omega)}$.
\end{enumerate}
\end{theorem*}
Here, by \emph{Lipschitz character} we refer to the quantitative data of a Lipschitz boundary atlas: the number of local Lipschitz charts needed to cover $\partial\Omega$, the maximal Lipschitz constant $L$ of the local Lipschitz graphs, and the minimal chart radius of the local Lipschitz graphs. We later discuss the dependence of the constant $\CHLZ$ on this Lipschitz character in Remark~\ref{rem:HLZ_extension_constant}.

While this extension successfully preserves high-order regularity, the family $\{\EHLZ^k\}_k$ lacks the crucial algebraic property of being a \emph{cochain map}, namely $d \circ \EHLZ^k \neq \EHLZ^{k+1} \circ d$. As observed in \cite[Remark~3.1]{hiptmair2012universalextension}, this identity holds locally but fails globally because of the partition-of-unity step in their construction.

In this paper, we introduce a graded family of extension operators $E^\bullet \coloneqq \{\cext{k}\}_{k=0}^n$ designed to restore this commutativity at the global level, while not losing the domain-independence of the stability constant. 
This hinges in particular on some results on Dirichlet Poincaré inequalities which we outline in Appendix~\ref{appendix:poincare}.
Another important detail concerning the domain-independence is that, with the quotient-norm normalisation adopted in Appendix~\ref{appendix:trace.scaling}, the weak tangential and normal trace constants for differential forms may be taken equal to the same universal value. We provide a self-contained account of this in Appendix~\ref{appendix:trace.scaling}.

A key feature of our construction is that it applies to bounded Lipschitz domains of arbitrary topology. More precisely, the only obstruction to strict commutation is the intrinsic cohomological obstruction carried by harmonic forms, and our theorem isolates this obstruction exactly by working on $\mathfrak H^k(\Omega)^\perp$.
Here, $\mathfrak H^k(\Omega)$ denotes the space of harmonic $k$-forms on $\Omega$, and orthogonality is understood with respect to the $L^2\Lambda^k(\Omega)$ inner product. See Section~\ref{sec:notations} for the definitions. 
Our main result is stated below. 

\begin{restatable}[Cochain extension operator]{theorem}{mainTheoremExtension}
\label{thm:global.extension}
Let $\Omega \subset \mathbb R^n$ be a bounded Lipschitz domain, and let
$K\subset\mathbb R^n$ be a bounded Lipschitz domain such that $\Omega\subseteq K$.
For any form degree $k\in\{0,\dots,n\}$, there exists a bounded linear operator
$E^k:H\Lambda^k(\Omega)\cap\mathfrak H^k(\Omega)^\perp\to H\Lambda^k(K)$ such that, for every
$\omega\in H\Lambda^k(\Omega)\cap\mathfrak H^k(\Omega)^\perp$,
\begin{enumerate}[label=(\roman*)]
\item $E^k\omega|_\Omega=\omega$,
\item $\|E^k\omega\|_{H\Lambda^k(K)}\le C_{\ext}\|\omega\|_{H\Lambda^k(\Omega)}$,
\item $d\,E^k\omega=E^{k+1}(d\omega)$ in $K$.
\end{enumerate}

Moreover, the operator can be constructed in either of the following two ways:
\begin{enumerate}[label=(\alph*)]
\item If $\Omega$ is strictly contained in $K$, then $E^k\omega$ has zero trace on $\partial K$.
Consequently, it extends by zero outside $K$ to a bounded operator
\[
\cext{k}:H\Lambda^k(\Omega)\cap\mathfrak H^k(\Omega)^\perp\to H\Lambda^k(\mathbb R^n),
\]
with the same stability bound. In this case, the constant $C_{\ext}$ depends on the
Lipschitz character of $\partial\Omega$, on $\dist(\partial\Omega,\partial K)^{-1}$, and on
the Dirichlet Poincaré constant of $K$. In particular,
if $K$ is convex, then this dependence may be expressed in terms of $\diam(K)$.

\item If no boundary condition is imposed on $\partial K$, then the assumption that
$\Omega$ be strictly contained in $K$ is not needed, provided that $K\setminus\overline\Omega$
is Lipschitz. In this case, the stability constant $C_{\ext}$ depends on the Lipschitz
characters of $\partial\Omega$ and $\partial K$, and on the mixed Poincaré constant of $K$. In particular, if $K$ is convex, then this dependence may
be expressed in terms of $\diam(K)$.
\end{enumerate}
\end{restatable}
In practice, the strict containment assumption in case (a) is not restrictive, since one may always replace $K$ by a slightly larger bounded Lipschitz domain containing $\Omega$.

The stability constant $C_\ext$ can also be specified to depend on the form degree $k$. Otherwise it is chosen as the maximum over the discrete set of all form degrees.

We highlight that the zero-trace condition in case (a) yields an extension to $\Real^n$ with compact support. Such compactly supported extensions play an important role, for instance, in the analysis of time-harmonic Maxwell equations; see, e.g., the wavenumber-explicit analysis in \cite{melenk2021wavenumber}, which relies on compactly supported lifting operators of Hiptmair, Li and Zou. More generally, our formulation in case (a) provides compact support together with the cochain map property, a combination that is important in structure-preserving formulations of partial differential equations, especially when differential constraints must be preserved by the extension. See, for instance, \cite[Lemma 1]{breit2025optimal} and \cite[Section 4.4.2]{puppi2022stabilized}, where one in particular needs that divergence-free vector fields are extended to divergence-free vector fields. Interestingly, these works refer to the HLZ extension in a way that implicitly treats the cochain map property as holding globally, even though, as noted above, this property fails in general because of the partition-of-unity step.

A useful way to interpret the orthogonality to harmonic forms, \emph{which is only required for property} (iii), is through de Rham cohomology. In the variant without boundary condition on $\partial K$, if an extension family $E^\bullet$ were a strict cochain map, meaning that $dE^k=E^{k+1}d$ for every $k$, and satisfied the extension property $(E^k\omega)|_\Omega=\omega$, then $E^\bullet$ would induce a map in cohomology $E^\sharp:H^k_{\mathrm{dR}}(\Omega)\to H^k_{\mathrm{dR}}(K)$, while the restriction map would induce $r^\sharp:H^k_{\mathrm{dR}}(K)\to H^k_{\mathrm{dR}}(\Omega)$. By construction one would have $r^\sharp\circ E^\sharp=\mathrm{id}$, so $r^\sharp$ must in particular be surjective. Thus a strict cochain extension can exist only if every cohomology class on $\Omega$ is the restriction of a class on $K$. In the zero-trace variant, the analogous discussion involves relative cohomology on $(K,\partial K)$.

When $K$ is contractible, this fails as soon as $H^k_{\mathrm{dR}}(\Omega)\neq 0$. In particular, a nontrivial harmonic form on $\Omega$, which represents a nonzero cohomology class, cannot admit a strict cochain extension into a contractible ambient domain. Consequently, the orthogonality in the datum to harmonic forms is not a defect of our construction, but rather the manifestation of an intrinsic topological obstruction. For this reason, we will refer to our construction simply as a \emph{cochain extension}, with the understanding that strict commutation holds on the orthogonal complement of the harmonic space.

To recover this cochain property, the core idea is to define the extension recursively over the form degrees, on the exterior region $K\setminus\overline\Omega$ through a first-order potential problem and to select, among all admissible solutions, the one of minimal $L^2$-norm. Although the admissibility conditions do not explicitly impose commutation with the exterior derivative, the recursive structure of the construction enforces it on the orthogonal complement of the harmonic space. In this way, the method remains genuinely global both analytically and topologically: it is not tied to collar neighbourhoods, nor to topologically trivial domains.

This minimal-norm recursive viewpoint is natural and is close in spirit to the discrete extension operator on local $k$-simplices constructed in \cite{falk2014local}. Other related work appears in \cite[Section 7.1]{licht2019smoothed}, where local cochain extensions are built from pullbacks over collar neighbourhoods for weakly Lipschitz domains. Such pullbacks commute naturally with the exterior derivative and satisfy useful boundedness properties, provided the underlying map is bi-Lipschitz. Their main limitation, however, is that the target domain is necessarily a collar neighbourhood of $\Omega$, and therefore has the same topology as $\Omega$. Moreover, these collar domains are localised near $\Omega$ and change geometrically when $\Omega$ changes. By contrast, our construction is not tied to a neighbourhood of $\Omega$ and allows for arbitrary ambient topologies. See Figure~\ref{fig:global..collar} for an illustrative comparison. 

While our framework restores the desired algebraic structure, it comes at the cost of reduced regularity: our analysis is restricted to the basic Sobolev spaces of the de Rham complex, $H\Lambda^k(\Omega)\coloneqq H^{(0,0)}\Lambda^k(\Omega)$, consisting of $L^2$ differential forms with $L^2$ exterior derivatives. This, however, is not a drawback for the applications we have in mind. On the contrary, these spaces provide the natural setting for mixed variational formulations, where preserving the cochain structure is more important than maintaining higher-order regularity.

The significance of restoring the cochain property is illustrated by the range of applications it unlocks, which we organise as follows.

Section~\ref{sec:cutfeec} is devoted to our primary motivation: providing a rigorous mathematical foundation for unfitted Finite Element Exterior Calculus (CutFEEC) \cite{erik_cutfeec}. This extension of the standard CutFEM \cite{burman2025cut,massing2014stabilized,burman2015cutfem,frachon2024darcy,frachon2024stokes} to differential forms follows the same basic unfitted framework. Unlike standard finite element methods, the computational mesh does not align with the physical boundary $\partial\Omega$. Instead, an active mesh $\Omega_h$ is constructed so as to strictly cover the domain, $\Omega \subset \Omega_h$, and therefore intersects the boundary arbitrarily rather than conforming to it. The fact that the discrete domain $\Omega_h$ strictly contains $\Omega$ poses significant mathematical challenges for the stability analysis of standard finite element schemes, especially for the Hodge--Laplace problem. Establishing stability on the active mesh $\Omega_h$ crucially requires discrete Poincaré inequalities whose constants are independent of the mesh size $h$ and of the specific active domain. While standard Poincaré inequalities hold on any fixed mesh $\Omega_h$, controlling the uniformity of the constant as $h\to0$ across a family of arbitrarily cut active domains is a major challenge. Our global cochain extension operator provides the exact analytical tool needed to transfer continuous stability bounds from the physical domain $\Omega$ to these irregular active meshes. The resulting statements are Theorem~\ref{thm:discrete.uniform.poincare} and Corollary~\ref{cor:stabilised_poincare}.

Section~\ref{sec:applications} shows that the impact of our construction extends beyond numerical analysis. The first application concerns \emph{uniform continuous Poincaré inequalities}. While optimal constants are known for convex domains \cite{payne1960optimalpoincare,acosta2004optimalpoincare}, uniform bounds are generally available only under strong topological restrictions \cite{ruiz2012uniformity,boulkhemair2007uniform}. By exploiting the global commutation of our extension family, Theorem~\ref{thm:uniform.poincare} establishes uniform $L^2$-Poincaré inequalities for differential forms of arbitrary degree on bounded Lipschitz domains, without any topological assumption.

The Poincaré inequality can be viewed as a lifting lemma for the exterior derivative. Hiptmair, Li, and Zou claimed such a result in \cite[Corollary 5.4]{hiptmair2012universalextension} without a global cochain property. However, it seems to us that there is a gap in the proof: it either identifies incompatible objects or overlooks the structural commutation with the exterior derivative that, from our point of view, must be preserved throughout the construction. See Remark~\ref{rem:lifting.lemma} for details. In this sense, our work also clarifies and repairs that argument in the $L^2$ setting.

Our second application concerns the Neumann eigenvalue problem for the coclosed Hodge Laplacian. In particular, we obtain a lower bound for the first nonzero eigenvalue on nonconvex domains $\Omega\subset\Real^n$,
\[
\lambda^{(k)}_1(\Omega)\ge\frac{C}{\diam(\mathrm{Conv}(\Omega_+))^4},
\]
where $\Omega_+\supset\Omega$ is a Lipschitz collar and $C$ depends only on the Lipschitz character of $\Omega$, the ambient dimension $n$, and the form degree $k$. The scaling in the diameter is not optimal, since it should be quadratic on convex domains, but the estimate holds for a very general class of domains of arbitrary topology. In terms of the reciprocal diameter of the convex hull, it generalises an estimate of Savo and Colbois \cite{colbois2021lower} for planar annuli when $k=1$. Moreover, since for nonconvex domains the first coclosed Neumann eigenvalue is not known, to our knowledge, to be monotone with respect to the form degree, it is noteworthy that all degrees share the same lower bound.

Finally, Section~\ref{sec:open.questions} discusses several open questions and possible directions for future work.

\section{Preliminaries and notation}
\label{sec:notations}
In this section, we collect the notation and functional-analytic preliminaries used throughout the paper. Unless otherwise stated, $\Omega\subset\mathbb R^n$ denotes a bounded Lipschitz domain. We let $\nu$ denote the outward normal vector to $\partial\Omega$. 
The same definitions as given in this section for $\Omega$ apply, with the domain replaced accordingly, to any other bounded Lipschitz domain appearing later, in particular to the ambient domain $K$ and, whenever it is Lipschitz, to the exterior domain $A\coloneqq K\setminus\overline\Omega$.

We adopt the standard convention that any space or operator associated with a differential form degree strictly smaller than $0$ or greater than $n$ is trivial.
In particular, $L^2\Lambda^{n+1}(\Omega)=\{0\}$, and operators such as $\cext{n+1}$ are understood to be zero.

Since the form degree always ranges over the finite set $\{0,\dots,n\}$, we systematically replace degree-dependent constants by their maximum over all admissible degrees. Thus, unless explicitly stated otherwise, all constants are understood to be uniform with respect to the form degree.

\paragraph*{Differential forms and Sobolev spaces.}
For $k\in\{0,\dots,n\}$, we denote by $L^2\Lambda^k(\Omega)$ the Hilbert space of square-integrable differential $k$-forms on $\Omega$, equipped with the inner product $(\omega,\eta)_{L^2\Lambda^k(\Omega)}\coloneqq\int_\Omega \omega\wedge\star\eta$, where $\star$ is the Hodge star operator.
We also use the standard Sobolev space $H^1\Lambda^k(\Omega)$ of differential $k$-forms with coefficients in $H^1(\Omega)$, endowed with the coefficientwise norm.
The exterior derivative is the densely defined operator $d:L^2\Lambda^k(\Omega)\to L^2\Lambda^{k+1}(\Omega)$ with domain $H\Lambda^k(\Omega)\coloneqq\{\omega\in L^2\Lambda^k(\Omega):d\omega\in L^2\Lambda^{k+1}(\Omega)\}$, equipped with the graph norm $\|\omega\|_{H\Lambda^k(\Omega)}^2\coloneqq\|\omega\|_{L^2\Lambda^k(\Omega)}^2+\|d\omega\|_{L^2\Lambda^{k+1}(\Omega)}^2$.

\paragraph*{Traces and codifferential.}
The tangential trace operator $\tr_{\partial\Omega}:H\Lambda^k(\Omega)\to H^{-1/2}\Lambda^k(\partial\Omega)$ is well defined and bounded, though only surjective onto its trace space $T^{-1/2}\Lambda^k(\partial\Omega)\coloneqq \tr_{\partial\Omega}(H\Lambda^k(\Omega))$, see e.g. \cite{buffa2002traces}. 
We write $H_0\Lambda^k(\Omega)\coloneqq\ker(\tr_{\partial\Omega})$ for the subspace of forms with vanishing tangential trace. In particular, for $k=n$ the tangential trace is trivial. The codifferential $\delta:L^2\Lambda^{k+1}(\Omega)\to L^2\Lambda^k(\Omega)$ is defined as the formal adjoint of $d:H_0\Lambda^k(\Omega)\to L^2\Lambda^{k+1}(\Omega)$. 
We set $H^*\Lambda^k(\Omega)\coloneqq\{\omega\in L^2\Lambda^k(\Omega):\delta\omega\in L^2\Lambda^{k-1}(\Omega)\}.$ 
A form is said to be coclosed if $\delta\omega=0$. 

The normal trace operator $\ntr_{\partial\Omega}:H^*\Lambda^k(\Omega)\to N^{-1/2}\Lambda^{k-1}(\partial\Omega)$ is defined by
$\star_{\partial\Omega}\ntr_{\partial\Omega}=\tr_{\partial\Omega}\star$, where $\star_{\partial\Omega}$ denotes the Hodge star on the $(n-1)$-dimensional boundary $\partial\Omega$.
It is a bounded linear operator onto its trace space $N^{-1/2}\Lambda^{k-1}(\partial\Omega)\coloneqq \ntr_{\partial\Omega}(H^*\Lambda^k(\Omega))$.
The $L^2$-pairing for smooth forms on $\partial\Omega$ extends to a bilinear pairing $T^{-1/2}\Lambda^k(\partial \Omega)\times N^{-1/2}\Lambda^{k}(\partial \Omega)\to \Real$ via taking a limit after using Green's formula, see for example \cite[Theorem 8]{weck2004traces}. The following equation, for $u\in H\Lambda^k(\Omega)$ and $v\in H^*\Lambda^{k+1}(\Omega)$, is the result:
\begin{align}\label{eq:weak.bdry.paring}
(du,v)_{L^2\Lambda^{k+1}(\Omega)}+(u,\delta v)_{L^2\Lambda^k(\Omega)} = \langle \tr_{\partial\Omega}u, \ntr_{\partial\Omega}v \rangle_{\partial\Omega}.
\end{align} 

\paragraph*{Harmonic forms, coexact forms, and Hodge decomposition.}
We recall the standard Hodge-theoretic framework for differential forms on Lipschitz domains; see, e.g., \cite[Chapters~2--3]{gunter2006hodge}.
We define the space of harmonic $k$-forms by
\[
\mathfrak H^k(\Omega)\coloneqq\{\omega\in H\Lambda^k(\Omega)\cap H^*\Lambda^k(\Omega):d\omega=0,\ \delta\omega=0,\ \ntr_{\partial\Omega}\omega=0\text{ on }\partial\Omega\},
\]
where, for sufficiently regular forms, $\ntr_{\partial\Omega}$ coincides with contraction with the outward unit normal vector field.
The $L^2$-orthogonal Hodge decomposition reads
\begin{equation}\label{eq:hodge.decomposition}
L^2\Lambda^k(\Omega)=\mathfrak H^k(\Omega)\oplus d(H\Lambda^{k-1}(\Omega))\oplus \ker(d|_\Omega)^\perp.
\end{equation}
In accordance with this decomposition, we will refer to forms in $\ker(d|_\Omega)^\perp$ as coexact forms. In particular, every coexact form is coclosed and satisfies the boundary condition $\ntr_{\partial\Omega}\omega=0$ on $\partial\Omega$.

We denote by $P_{\mathrm{ex}}^\Omega:L^2\Lambda^k(\Omega)\to d(H\Lambda^{k-1}(\Omega))$ the $L^2$-orthogonal projection onto the subspace of exact forms.
For $\omega\in L^2\Lambda^k(\Omega)$, we let $\alpha_\perp(\omega)\in H\Lambda^{k-1}(\Omega)$ denote the unique minimal-norm potential in $(\ker(d|_\Omega))^\perp$ satisfying $d\alpha_\perp(\omega)=P_{\mathrm{ex}}^\Omega\omega$. By the Poincaré inequality on $(\ker(d|_\Omega))^\perp$, this operator is bounded from $L^2\Lambda^k(\Omega)$ to $H\Lambda^{k-1}(\Omega)$.
Thus, by \eqref{eq:hodge.decomposition}, we may write
\begin{align}\label{eq:form.version.hodgedecomp}
\omega = d\alpha_\perp(\omega) + \beta_\perp + q, \qquad \beta_\perp\in \ker(d|_\Omega)^\perp,\ q\in\mathfrak H^k(\Omega).
\end{align}
If in addition $\omega\in H\Lambda^k(\Omega)$, then $d\beta_\perp=d\omega$, and uniqueness of the minimal-norm potential associated with the datum $d\omega$ implies that $\beta_\perp=\alpha_\perp(d\omega)$.

\paragraph*{Exterior-domain spaces and harmonic fields.}
Let $A\coloneqq K\setminus\overline\Omega$ denote the exterior domain. For the zero-trace
variant of the construction, we use
\begin{align*}
H_0\Lambda^k(A)&\coloneqq \{\mu\in H\Lambda^k(A):\tr_{\partial A}\mu=0\},
\\
\mathfrak H_0^k(A)&\coloneqq
\{q\in H_0\Lambda^k(A)\cap H^*\Lambda^k(A) : dq=0,\ \delta q=0\}.
\end{align*}

For the mixed-boundary variant, where a tangential trace is prescribed only on
$\partial\Omega$, we use
\begin{align*}
H_{\partial\Omega,0}\Lambda^k(A)
&\coloneqq
\{\mu\in H\Lambda^k(A):\tr_{\partial\Omega}\mu=0\}, \\
\mathfrak H_{\mathrm{mix}}^k(A)&\coloneqq 
\{q\in H_{\partial\Omega,0}\Lambda^k(A)\cap H^*\Lambda^k(A) : dq=0,\ \delta q=0,\ \ntr_{\partial K}q=0\}.
\end{align*}

\section{Construction of the cochain extension operators}\label{sec:cochain.extensions}
Let $\Omega \subset \Real^n$ be a bounded Lipschitz domain. In this section, we prove Theorem~\ref{thm:global.extension} by constructing the graded cochain extension family recursively on the exterior domain.

We proceed in two steps, corresponding to the two variants in Theorem~\ref{thm:global.extension}. We first treat the zero-trace construction on $\partial K$, which yields the extension of part~(a) and, by extension by zero, a bounded operator into $\Real^n$. We then explain how the same recursive construction adapts to the mixed-boundary setting relevant to part~(b), where the condition on $\partial K$ is removed and the relative harmonic fields are replaced by the corresponding mixed harmonic fields.

Since the argument is most transparent in the zero-trace case, we present that construction in full detail first. Throughout this part, we therefore assume, without loss of generality, that $\Omega$ is strictly enclosed in a bounded domain $K$, and we write
\[
A= K\setminus\overline\Omega.
\]
The modifications needed for the mixed-boundary variant will be described after the proof of part~(a).

\subsection{Recursive definition and well-posedness}
Let $A=K\setminus\overline\Omega$ and let $\mathfrak H_0^k(A)$ be as in Section~\ref{sec:notations}.
We construct the graded cochain extension family $E^\bullet$ by descending induction on the form degree.

For $\omega\in H\Lambda^k(\Omega)\cap\mathfrak H^k(\Omega)^\perp$, we set
\begin{equation}
\label{eq:ext.def}
\cext{k}\omega\coloneqq
\begin{cases}
\omega,&\text{in }\Omega,\\
\lambda,&\text{in }A,
\end{cases}
\end{equation}
where $\lambda\in H\Lambda^k(A)$ (or $\lambda\in L^2\Lambda^n(A)$ if $k=n$) is defined below.

Before distinguishing the cases $k=n$ and $k<n$, we first introduce the harmonic coefficients. Let $\alpha_\perp(\omega)\in H\Lambda^{k-1}(\Omega)$ denote the minimal-norm potential associated with the datum $\omega$, as introduced in Section~\ref{sec:notations}. For every $q\in \mathfrak H_0^k(A)$, we define
\begin{equation*}%\label{eq:def.ch.cont}
c_q^k(\omega)\coloneqq \langle \tr_{\partial\Omega}\alpha_\perp(\omega), \ntr_{\partial\Omega} q \rangle_{\partial\Omega}.
\end{equation*}
This pairing is well defined by equation \eqref{eq:weak.bdry.paring}: note that $\tr_{\partial\Omega}\alpha_\perp(\omega)$ and $\ntr_{\partial\Omega}q$ are both $(k-1)$-forms on $\partial\Omega$. 
In particular, the coefficients $c_q^k=c_q^k(\omega)$ depend linearly on $\omega$. By convention, for $k=0$ we set $c_q^0=0$.

\smallskip
\noindent
\emph{Case $k=n$.}
We define $\lambda\in \mathfrak H_0^n(A)\subset L^2\Lambda^n(A)$ to be the minimal-norm solution of
\begin{equation}
\label{eq:top.degree.def}
(\lambda,q)_{L^2\Lambda^n(A)}=c_q^n
\qquad\forall q\in \mathfrak H_0^n(A).
\end{equation}
Equivalently, $\lambda$ is the unique element of $\mathfrak H_0^n(A)$ representing, by Riesz's theorem on the finite-dimensional space $\mathfrak H_0^n(A)$, the functional $q\mapsto c_q^n$.

\smallskip
\noindent
\emph{Case $k<n$.}
We first note that for every $\omega\in H\Lambda^k(\Omega)$ one has $d\omega\in H\Lambda^{k+1}(\Omega)$, since $d^2=0$. Assuming that $\cext{k+1}$ has already been defined, and given $\omega\in H\Lambda^k(\Omega)\cap\mathfrak H^k(\Omega)^\perp$, we define $\lambda\in H\Lambda^k(A)$ to be the minimal-norm solution of the first-order problem
\begin{equation}
\label{eq:recursive.def.cont}
\left\{
\begin{aligned}
d\lambda&=\cext{k+1}(d\omega)&&\text{in }A,\\
\tr_{\partial\Omega}\lambda&=\tr_{\partial\Omega}\omega&&\text{on }\partial\Omega,\\
\tr_{\partial K}\lambda&=0&&\text{on }\partial K,\\
(\lambda,q)_{L^2\Lambda^k(A)}&=c_q^k&&\text{for all }q\in\mathfrak H_0^k(A).
\end{aligned}
\right.
\end{equation}

The next lemma shows that this recursive definition is well posed.

\begin{lemma}[Well-posedness of the recursive construction]
\label{lem:well.posedness.cochain.extension}
For every $k\in\{0,\dots,n\}$ and every
$\omega\in H\Lambda^k(\Omega)\cap\mathfrak H^k(\Omega)^\perp$, the recursive construction defining $\cext{k}\omega$ is well posed.
In particular, for $k\in\{0,\dots,n-1\}$ the problem \eqref{eq:recursive.def.cont} admits a unique minimal-norm solution
$\lambda\in H\Lambda^k(A)$, and therefore the graded family $E^\bullet$ is well defined.
\end{lemma}
\begin{proof}
We argue by descending induction on $k$, with base case $k=n$.

At top degree, the extension $\cext{n}\omega$ is well defined by \eqref{eq:top.degree.def}.

Now let $k<n$ and assume that $\cext{k+1}$ is well defined. Set $\xi\coloneqq \cext{k+1}(d\omega)|_A$, and let $\theta\in H\Lambda^k(\partial A)$ be the boundary datum given by $\theta=\tr_{\partial\Omega}\omega$ on $\partial\Omega$ and $\theta=0$ on $\partial K$. By standard Hodge--de Rham theory for differential forms with boundary conditions (see, e.g., \cite[Chapter~3, Theorem 3.1.1]{gunter2006hodge}), the system $d\lambda=\xi$ in $A$ with $\tr_{\partial A}\lambda=\theta$ admits a solution if and only if $\xi$ is closed, the traces are compatible, and for every $q\in\mathfrak H_0^{k+1}(A)$ one has
\begin{equation}\label{eq:compatibility.cont}
(\xi,q)_{L^2\Lambda^{k+1}(A)} = \langle\theta,\ntr_{\partial A}q\rangle_{\partial A}.
\end{equation}

The first condition follows from the inductive hypothesis, since $d\xi=d(\cext{k+1}(d\omega))=\cext{k+2}(d^2\omega)=0$. The trace compatibility is also immediate: on $\partial K$ both sides vanish by construction, while on $\partial\Omega$ we have $\tr_{\partial\Omega}\xi=\tr_{\partial\Omega}(d\omega)=d_{\partial\Omega}(\tr_{\partial\Omega}\omega)=d_{\partial\Omega}\theta$.

It remains to verify \eqref{eq:compatibility.cont}. Let $q\in\mathfrak H_0^{k+1}(A)$. Since $\xi=\cext{k+1}(d\omega)|_A$ is defined at degree $k+1$, we have $(\xi,q)_{L^2(A)}=c_q^{k+1}(d\omega)$. By definition of the coefficient $c_q^{k+1}$, the potential is precisely $\alpha_\perp(d\omega)$. Since $\omega\in\mathfrak H^k(\Omega)^\perp$, the Hodge decomposition \eqref{eq:form.version.hodgedecomp} yields $\omega=d\alpha_\perp+\beta_\perp$ for $\alpha_\perp=\alpha_\perp(\omega)\in H\Lambda^{k-1}(\Omega)$ and $\beta_\perp=\alpha_\perp(d\omega)\in(\ker d|_\Omega)^\perp$.
%Because $d\omega$ is exact, Section~\ref{sec:notations} shows that $\alpha_\perp(d\omega)=\beta_\perp$, hence $c_q^{k+1}=\langle\tr_{\partial\Omega}\beta_\perp,\ntr_{\partial\Omega}q\rangle_{\partial\Omega}$.
Hence $c_q^{k+1}=\langle\tr_{\partial\Omega}\beta_\perp,\ntr_{\partial\Omega}q\rangle_{\partial\Omega}$.

Therefore,
\[
\langle\theta,\ntr_{\partial A}q\rangle_{\partial A}
=\langle\tr_{\partial\Omega}\omega,\ntr_{\partial\Omega}q\rangle_{\partial\Omega}
=\langle\tr_{\partial\Omega}(d\alpha_\perp),\ntr_{\partial\Omega}q\rangle_{\partial\Omega}
+\langle\tr_{\partial\Omega}\beta_\perp,\ntr_{\partial\Omega}q\rangle_{\partial\Omega}.
\]
The second term is exactly $c_q^{k+1}$. For the first one, note that $\tr_{\partial\Omega}\alpha_\perp\in T^{-1/2}\Lambda^{k-1}(\partial\Omega)$. Extending this datum by zero on $\partial K$, thus obtaining an element of $T^{-1/2}\Lambda^{k-1}(\partial A)$, surjectivity of the tangential trace on $A$ yields $\alpha_A\in H\Lambda^{k-1}(A)$ such that $\tr_{\partial\Omega}\alpha_A=\tr_{\partial\Omega}\alpha_\perp$ and $\tr_{\partial K}\alpha_A=0$.
Since the exterior derivative commutes with the tangential trace, we have $\tr_{\partial\Omega}(d\alpha_\perp)=\tr_{\partial\Omega}(d\alpha_A)$ and $\tr_{\partial K}(d\alpha_A)=0$. Hence
\[
\langle \tr_{\partial\Omega}d\alpha_\perp,\ntr_{\partial\Omega}q\rangle_{\partial\Omega}=\langle \tr_{\partial A}d\alpha_A,\ntr_{\partial A}q\rangle_{\partial A}.
\]
Thus, Green's formula \eqref{eq:weak.bdry.paring} gives
\[
\langle \tr_{\partial\Omega}d\alpha_\perp,\ntr_{\partial\Omega}q\rangle_{\partial\Omega}=(d(d\alpha_A), q)_{L^2\Lambda^{k+1}(A)}+(d\alpha_A, \delta q)_{L^2\Lambda^{k}(A)} = 0,
\]
since $d^2=0$ and $q\in\mathfrak H_0^{k+1}(A)$ satisfies $\delta q=0$ in $A$. We conclude that $\langle\theta,\ntr_{\partial A}q\rangle_{\partial A}=c_q^{k+1}=(\xi,q)_{L^2(A)}$, which proves \eqref{eq:compatibility.cont}. 
Therefore, the system \eqref{eq:recursive.def.cont} admits at least one solution.

To conclude, note that the set of all solutions to \eqref{eq:recursive.def.cont} is a closed affine subspace of $L^2\Lambda^k(A)$. As such, it contains a unique element of minimal $L^2$-norm. This proves the existence and uniqueness of the minimal-norm solution $\lambda$.
\end{proof}

Because the recursive construction imposes the zero tangential trace condition $\tr_{\partial K}\lambda=0$ on the exterior field $\lambda$, the form $\cext{k}\omega\in H\Lambda^k(K)$ defined in \eqref{eq:ext.def} extends by zero outside $K$ to a well-defined element of $H\Lambda^k(\mathbb R^n)$.

\begin{definition}[Extension operator]
For any $\omega\in H\Lambda^k(\Omega)\cap\mathfrak H^k(\Omega)^\perp$, let $\cext{k}\omega\in H\Lambda^k(K)$ be the recursively defined extension from \eqref{eq:ext.def}. When $\Omega$ is strictly contained in $K$, we use the same notation $\cext{k}\omega$ for its extension by zero to $\mathbb R^n$. In this way,
\[
\cext{k}:H\Lambda^k(\Omega)\cap\mathfrak H^k(\Omega)^\perp\to H\Lambda^k(\mathbb R^n).
\]
\end{definition}

\subsection{Stability bounds and proof of the main theorem}
We now turn to the proof of Theorem~\ref{thm:global.extension}.
The first step is a stability result for solutions to problem \eqref{eq:recursive.def.cont}.

\begin{lemma}[Stability of the recursive Dirichlet step]\label{lem:level.k.stab.bound}
Let $\lambda\in H\Lambda^k(A)$ be the minimal-norm solution to
\eqref{eq:recursive.def.cont} with datum
$\omega\in H\Lambda^k(\Omega)\cap \mathfrak H^k(\Omega)^\perp$. Then there exists a constant $\tilde C>0$ such that
\begin{equation}\label{eq:one.step.stability}
\|\lambda\|_{L^2\Lambda^k(A)}
\le
C_{P,0}(A)\,\|\cext{k+1}(d\omega)\|_{L^2\Lambda^{k+1}(A)}
+\tilde C\,\|\omega\|_{H\Lambda^k(\Omega)}.
\end{equation}
Moreover, if $\omega=d\eta$ for some $\eta\in H\Lambda^{k-1}(\Omega)$, then
\begin{equation}\label{eq:one.step.stability.d}
\|\lambda\|_{L^2\Lambda^k(A)}
\le \tilde C\,\|d\eta\|_{L^2\Lambda^k(\Omega)}.
\end{equation}
\end{lemma}
\begin{proof}
To estimate $\lambda$, we first choose a lifting of the boundary datum and then correct it so as to satisfy the differential and harmonic constraints in \eqref{eq:recursive.def.cont}. Any lifting with the required trace and suitable stability bounds would do. For definiteness, we take a cutoff of the HLZ-extension.

Set $\rho\coloneqq \dist(\partial\Omega,\partial K)>0$, and choose $\varphi_0\in C_c^\infty(K)$ such that $0\le \varphi_0\le 1$, $\varphi_0|_{\overline\Omega}=1$, and $\|\nabla\varphi_0\|_{L^\infty(K)}\lesssim \rho^{-1}$. Let $\hat\omega_0\coloneqq \varphi_0\,\EHLZ^k\omega$. Then $\tr_{\partial\Omega}\hat\omega_0=\tr_{\partial\Omega}\omega$, $\tr_{\partial K}\hat\omega_0=0$, and $\|\hat\omega_0\|_{L^2\Lambda^k(A)}\le \CHLZ\|\omega\|_{H\Lambda^k(\Omega)}$. Moreover,
\begin{equation*}%\label{eq:cutoff.d.bound}
\|d\hat\omega_0\|_{L^2\Lambda^{k+1}(A)}
\lesssim
(1+\rho^{-1})\,\CHLZ\,\|\omega\|_{H\Lambda^k(\Omega)}.
\end{equation*}

We note that $d\lambda=\cext{k+1}(d\omega)|_A$ holds by definition of problem \eqref{eq:recursive.def.cont}. 
Thus $\zeta\coloneqq \cext{k+1}(d\omega)|_A-d\hat\omega_0=d(\lambda-\hat\omega_0)$. Moreover, $\lambda-\hat\omega_0\in H_0\Lambda^k(A)$, because $\lambda$ and $\hat\omega_0$ have the same tangential trace on $\partial\Omega$ and both vanish on $\partial K$. Hence $\zeta\in d(H_0\Lambda^k(A))$.

By the relative Hodge decomposition on $A$, there therefore exists a unique $\gamma_\perp\in H_0\Lambda^k(A)\cap(\ker d|_A)^\perp$ such that $d\gamma_\perp=\zeta$. The Dirichlet Poincaré inequality on $A$ yields
\begin{align}\nonumber
\|\gamma_\perp\|_{L^2\Lambda^k(A)}
&\le
C_{P,0}(A)\Bigl(
\|\cext{k+1}(d\omega)\|_{L^2\Lambda^{k+1}(A)}
+
\|d\hat\omega_0\|_{L^2\Lambda^{k+1}(A)}
\Bigr) \\
&\le C_{P,0}(A)(1+\rho^{-1})\,\CHLZ\,\|\omega\|_{H\Lambda^k(\Omega)} + C_{P,0}(A)\,\|\cext{k+1}(d\omega)\|_{L^2\Lambda^{k+1}(A)}.\label{eq:gamma.bound}
\end{align}

Next, let $h\in\mathfrak H_0^k(A)$ be the unique harmonic form such that $(h,q)_{L^2(A)}=c_q^k-(\hat\omega_0,q)_{L^2(A)}$ for all $q\in\mathfrak H_0^k(A)$. Then $h$ is the Riesz representative of the functional
$q\mapsto \langle \tr_{\partial\Omega}\alpha_\perp(\omega),\ntr_{\partial\Omega}q\rangle_{\partial\Omega}-(\hat\omega_0,q)_{L^2(A)}$.
Thus, by Cauchy--Schwarz, the weak normal trace estimate on $A$, and the weak tangential trace estimate on $\Omega$ from Appendix \ref{appendix:trace.scaling},
\begin{align*}
\|h\|_{L^2\Lambda^k(A)} &= \max_{q\in \mathfrak H_0^k(A)} \frac{\langle \tr_{\partial\Omega}\alpha_\perp(\omega), \ntr_{\partial\Omega} q \rangle_{\partial\Omega} - (\hat\omega_0,q)_{L^2\Lambda^k(A)}}{\|q\|_{L^2\Lambda^k(A)}} \nonumber \\
&\leq C_{\mathrm{tr}}(A)\,\|\tr_{\partial\Omega}\alpha_\perp(\omega)\|_{T^{-1/2}\Lambda^{k-1}(\partial\Omega)} + \|\hat\omega_0\|_{L^2\Lambda^k(A)} \nonumber \\
&\leq \bigl(C_{\mathrm{tr}}(A)\,C_{\mathrm{tr}}(\Omega)\,\|\alpha_\perp\|_{\mathrm{op}} + \CHLZ\bigr)\,\|\omega\|_{H\Lambda^k(\Omega)},
\end{align*}
where $\|\alpha_\perp\|_{\mathrm{op}}$ is the operator norm of $\alpha_\perp:L^2\Lambda^k(\Omega)\to H\Lambda^{k-1}(\Omega)$.

Define 
\[
\lambda^\ast\coloneqq \hat\omega_0+\gamma_\perp+h. 
\]
By construction, $d\lambda^\ast=\cext{k+1}(d\omega)|_A$, its traces agree with those prescribed in \eqref{eq:recursive.def.cont}, and, since $\gamma_\perp\perp \mathfrak H_0^k(A)$, its harmonic moments are exactly $c_q^k$.
Therefore $\lambda^\ast$ is an admissible solution of \eqref{eq:recursive.def.cont}. Since $\lambda$ is the minimal-norm solution, $\|\lambda\|_{L^2\Lambda^k(A)}\le \|\lambda^\ast\|_{L^2\Lambda^k(A)}$.
Combining these bounds, and absorbing into $\tilde C$ all terms multiplying $\|\omega\|_{H\Lambda^k(\Omega)}$, we obtain \eqref{eq:one.step.stability}.

If $\omega=d\eta$ is exact, then $d\omega=d^2\eta=0$, and therefore $\cext{k+1}(d\omega)=0$. Thus \eqref{eq:one.step.stability} immediately yields \eqref{eq:one.step.stability.d}.
\end{proof}

\begin{remark}[Dependence of the constant]\label{rem:one.step.constant}
Inspecting the proof of Lemma~\ref{lem:level.k.stab.bound}, one sees that $\tilde C$ may be chosen depending only on $\CHLZ$, on the Dirichlet Poincaré constant $C_{P,0}(A)$, on the trace constants of $\Omega$ and $A$, on the operator norm of $\alpha_\perp$, and on $\rho^{-1}=\dist(\partial\Omega,\partial K)^{-1}$ through the cutoff estimate for $\hat\omega_0$. More precisely, up to the hidden constant in \eqref{eq:gamma.bound}, one may choose $\tilde C$ as
\[
\tilde C = \bigl(1+C_{P,0}(A)(1+\rho^{-1})\bigr)\CHLZ + C_{\mathrm{tr}}(A)\,C_{\mathrm{tr}}(\Omega)\,\|\alpha_\perp\|_{\mathrm{op}}.
\]

Apriori, the constant depends on the form degree $k$, $\tilde C = \tilde C(k)$, but one may of course take the maximum of the discrete set of all form degrees if needed. 

If one wishes to express the dependence in terms of the ambient set $K$, one may further bound $C_{P,0}(A)\le C_{P,0}(K)$ by extension by zero from $A$ into $\Omega$. In particular, if $K$ is convex, this yields an upper bound in terms of $\diam(K)/\pi$. See Appendix~\ref{appendix:poincare} and Lemma~\ref{lem:dirichlet.poincare.monotone} therein.

Lastly, by Appendix~\ref{appendix:trace.scaling}, with the quotient-norm normalisation adopted there one may in fact take $C_{\mathrm{tr}}(A)=C_{\mathrm{tr}}(\Omega)=\sqrt2$.
\end{remark} 
\medskip

We can now prove Theorem~\ref{thm:global.extension}.
\mainTheoremExtension*

\begin{proof}
We first prove part (a). Since the recursive construction imposes the condition $\tr_{\partial K}\lambda=0$ on the exterior field, the extension by zero outside $K$ defines an element of $H\Lambda^k(\mathbb R^n)$ with the same norm. It therefore suffices to prove the result on $K$.

Property (i) is immediate from the piecewise definition of $\cext{k}\omega$. Moreover, on the exterior domain $A$ the recursive problem gives $d\lambda=\cext{k+1}(d\omega)$. Since the trace matching condition $\tr_{\partial\Omega}\lambda=\tr_{\partial\Omega}\omega$ is built into the definition, the piecewise form $\cext{k}\omega$ has no distributional jump across $\partial\Omega$. Hence its global weak exterior derivative satisfies $d(\cext{k}\omega)=\cext{k+1}(d\omega)$ in $K$, which proves (iii).

It remains to prove the stability bound~(ii). By the convention fixed in Section~\ref{sec:notations}, we work with constants that are uniform with respect to the form degree.

We first consider the case $k=n$. Since $H\Lambda^n(K)=L^2\Lambda^n(K)$, we have
$\|\cext{n}\omega\|_{H\Lambda^n(K)} \le \|\omega\|_{L^2\Lambda^n(\Omega)} + \|\lambda\|_{L^2\Lambda^n(A)}$, where $\lambda=(\cext{n}\omega)|_A$.
By \eqref{eq:top.degree.def}, $\lambda\in\mathfrak H_0^n(A)$ is the Riesz representative of the functional $q\mapsto c_q^n$, and arguing exactly as in the estimate of the harmonic correction in Lemma~\ref{lem:level.k.stab.bound}, we obtain
$\|\lambda\|_{L^2\Lambda^n(A)}\leq C_{\mathrm{tr}}(A)\,C_{\mathrm{tr}}(\Omega)\,\|\alpha_\perp\|_{\mathrm{op}}\,\|\omega\|_{H\Lambda^n(\Omega)}$.
Hence, after enlarging $\tilde C$ if necessary, we have
$\|\cext{n}\omega\|_{H\Lambda^n(K)}\le \tilde C\,\|\omega\|_{H\Lambda^n(\Omega)}$.

Assume now $k<n$, and let $\lambda=\cext{k}\omega|_A$ be the minimal-norm solution of \eqref{eq:recursive.def.cont}. Since $\cext{k}\omega$ equals $\omega$ on $\Omega$ and $\lambda$ on $A$, while $d\lambda=\cext{k+1}(d\omega)|_A$, we have
\begin{equation}\label{eq:EkH-split}
\|\cext{k}\omega\|_{H\Lambda^k(K)}
\le
\|\omega\|_{H\Lambda^k(\Omega)}
+
\|\lambda\|_{L^2\Lambda^k(A)}
+
\|d\lambda\|_{L^2\Lambda^{k+1}(A)}.
\end{equation}
By inequality \eqref{eq:one.step.stability} of Lemma~\ref{lem:level.k.stab.bound},
$\|\lambda\|_{L^2\Lambda^k(A)} \le C_{P,0}(A)\,\|d\lambda\|_{L^2\Lambda^{k+1}(A)} + \tilde C\,\|\omega\|_{H\Lambda^k(\Omega)}$.
Moreover, since $d\lambda=\cext{k+1}(d\omega)|_A$ and $d^2\omega=0$, the improved inequality \eqref{eq:one.step.stability.d} of the same lemma applied at degree $k+1$ to the datum $d\omega$ gives
$\|d\lambda\|_{L^2\Lambda^{k+1}(A)} \le \tilde C\,\|d\omega\|_{L^2\Lambda^{k+1}(\Omega)}$.
Substituting these bounds into \eqref{eq:EkH-split}, we obtain
\[
\|\cext{k}\omega\|_{H\Lambda^k(K)}
\le
(1+\tilde C)\,\|\omega\|_{H\Lambda^k(\Omega)}
+
(1+C_{P,0}(A))\,\tilde C\,\|d\omega\|_{L^2\Lambda^{k+1}(\Omega)}
\le
\bigl(1+(2+C_{P,0}(A))\,\tilde C\bigr)\,\|\omega\|_{H\Lambda^k(\Omega)}.
\]

This proves~(ii) with
\[
C_{\ext}\coloneqq 1+(2+C_{P,0}(A))\,\tilde C
\qquad\text{for all }k\in\{0,\dots,n\},
\]
after enlarging $\tilde C$ once and for all, if necessary, to include the top-degree case.

If $K$ is convex, then the dependence of $C_{\ext}$ on $\diam(K)$ follows from Remark~\ref{rem:one.step.constant} together with the corresponding bound on the Dirichlet Poincaré constant of $K$.

For part~(b), one considers the modified recursive problem obtained by removing the condition $\tr_{\partial K}\lambda=0$. In that case, the zero-trace setting on $A$ is replaced throughout by the corresponding mixed-boundary one: the space $H_0\Lambda^k(A)$ is replaced by $H_{\partial\Omega,0}\Lambda^k(A)$, the relative harmonic fields $\mathfrak H_0^k(A)$ are replaced by the mixed harmonic fields $\mathfrak H_{\mathrm{mix}}^k(A)$, and the orthogonality condition in the recursive problem is imposed with respect to $\mathfrak H_{\mathrm{mix}}^k(A)$. Likewise, in the well-posedness argument, the compatibility condition is written against test fields in $\mathfrak H_{\mathrm{mix}}^{k+1}(A)$, so that the boundary pairing involves only $\partial\Omega$, while the contribution on $\partial K$ vanishes by the condition $\ntr_{\partial K}q=0$.

With these substitutions, the arguments of Lemmas~\ref{lem:well.posedness.cochain.extension} and \ref{lem:level.k.stab.bound} apply in the same way. In particular, one may now take the HLZ-extension itself as the lifting, without introducing the cutoff $\varphi_0$, and therefore the dependence on $\rho^{-1}=\dist(\partial\Omega,\partial K)^{-1}$ disappears.
If $\dist(\partial\Omega,\partial K)>0$, then $A=K\setminus\overline\Omega$ is automatically Lipschitz, so the same trace and Hodge-theoretic framework as in part~(a) applies. If $\dist(\partial\Omega,\partial K)=0$, we assume in addition that $A$ is Lipschitz, exactly as in the statement of part~(b), so that the mixed trace and Hodge-theoretic framework on $A$ remains available. The resulting stability constant then depends only on the Lipschitz characters of $\partial\Omega$ and $\partial K$, and on $\diam(K)$ if $K$ is chosen convex.
\end{proof}

\begin{remark}[Exact extension improved bound]\label{rem:exact.extension.improved.bound}
When the datum $\omega$ is exact, the proof of Theorem~\ref{thm:global.extension}~(ii)
yields the sharper estimate
\[
\|\cext{k}\omega\|_{H\Lambda^k(K)}\le(1+\tilde C)\,\|\omega\|_{H\Lambda^k(\Omega)}.
\]
Indeed, for $k<n$, if $\omega=d\eta$, then $d\omega=0$, so the improved bound \eqref{eq:one.step.stability.d} in Lemma~\ref{lem:level.k.stab.bound} gives $\|\lambda\|_{L^2\Lambda^k(A)}\le \tilde C\,\|\omega\|_{L^2\Lambda^k(\Omega)}$ and $d\lambda=0$, whence the above estimate follows from the splitting of the $H\Lambda^k(K)$-norm used in the proof of the theorem.
For $k=n$, the same bound holds after enlarging $\tilde C$, if necessary, to absorb the top-degree estimate.
In particular, in applications where the datum is exact, one may replace $C_{\ext}$ by $1+\tilde C$, so that only one occurrence of the Poincaré constant remains.
\end{remark}

\subsection{Topological considerations and further remarks}
We next collect a few remarks on the construction and its scope.

\begin{remark}[Dependence of the HLZ stability constant]
\label{rem:HLZ_extension_constant}
We emphasise that the stability constant $\CHLZ$ of the HLZ-extension operator depends only on the quantitative Lipschitz character of $\Omega$. This is consistent with the behaviour of Stein's classical extension operator for scalar Sobolev spaces on Lipschitz domains; see \cite[Theorem~5, p.~181]{stein1970singular} and \cite[Theorem~1.3.2]{rogers2004extensionthesis}.

Although this dependence is not explicitly discussed in \cite{hiptmair2012universalextension}, it follows directly from their construction. For a local Lipschitz epigraph, the continuity bounds in \cite[Theorem~3.5]{hiptmair2012universalextension} rely on two ingredients: an averaging kernel $\psi$ and a regularised distance function $\delta^*$. The kernel $\psi$ is given by an explicit expression on $\mathbb R^n$ and is therefore independent of the geometry of the domain. On the other hand, the regularised distance function $\delta^*$ satisfies local bounds that scale linearly with the distance to the boundary, with proportionality constants depending only on the Lipschitz constant of the local graph; see \cite[Lemma~3.1]{hiptmair2012universalextension}.

Passing from the local construction to a general bounded Lipschitz domain by means of a partition of unity, as in \cite[Theorem~3.6]{hiptmair2012universalextension}, introduces dependence on the local chart radii through the corresponding partition functions, for instance through estimates of the form $|D\chi_i|\sim r_i^{-1}$. In particular, $\CHLZ$ does not carry any additional dependence on the global diameter $\diam(\Omega)$ beyond that already encoded in the chosen boundary atlas. 
\end{remark}\medskip

\begin{remark}[Ambient topology and the extension of harmonic forms]\label{rem:ambient.topology}
The orthogonality assumption in Theorem~\ref{thm:global.extension} is needed in general, but it is only forced by the possible topological obstruction discussed in the introduction. Indeed, if $\omega\in\mathfrak H^k(\Omega)$, then $d\omega=0$ and $P_{\mathrm{ex}}^\Omega\omega=0$, so that the recursive problem reduces to the search for a closed exterior field $\lambda$ on $A$ with the prescribed trace on $\partial\Omega$, the prescribed boundary condition on $\partial K$, and vanishing harmonic constraints.

Thus, for harmonic data, the issue is no longer the cochain relation itself, but rather whether the cohomology class of $\omega$ admits an extension to the ambient domain compatible with the boundary condition imposed on $\partial K$. In the variant with vanishing tangential trace on $\partial K$, this means that $[\omega]\in H^k_{\mathrm{dR}}(\Omega)$ must lie in the image of the restriction map $H^k_{\mathrm{dR}}(K,\partial K)\to H^k_{\mathrm{dR}}(\Omega)$.
In the variant without boundary condition on $\partial K$, the corresponding condition is that $[\omega]$ lie in the image of $H^k_{\mathrm{dR}}(K)\to H^k_{\mathrm{dR}}(\Omega)$.
In this case, the obstruction disappears, for example, when $\Omega$ is a deformation retract of $K$. In the zero-trace variant, the corresponding obstruction is instead expressed in terms of the relative cohomology of the pair $(K,\partial K)$.

We stress, however, that Theorem~\ref{thm:global.extension} does not address the extension of harmonic forms. Its statement is formulated on $\mathfrak H^k(\Omega)^\perp$, which is the maximal subspace on which the cochain map property is always available without any additional topological assumption on the bounding domain. 
\end{remark}\medskip

\begin{remark}[The cochain map property is energy minimising if $K$ and $\Omega$ are topologically equivalent]
Let $\Omega$ be a deformation retract of $K$ and consider as extensions of $\omega\in H\Lambda^k(\Omega)$ the solutions to the following nested minimisation problem 
\begin{equation*}
\min\Bigl\{\|\mu\|_{L^2\Lambda^k(K)}^2 : \mu \text{ solves } \min_{\tilde\mu\in\mathcal V_\omega^k}\|d\tilde\mu\|_{L^2\Lambda^{k+1}(K)}^2
\Bigr\},
\end{equation*}
where the admissible set is $\mathcal V_\omega^k = \{\lambda\in H\Lambda^k(K) : \lambda|_\Omega=\omega\}$. 
It is not difficult to see that the set of solutions to the inner problem is convex and closed in $L^2\Lambda^k(K)$, whence there is a unique solution which will satisfy a nice bound thanks to the Poincaré inequality and the fact that the HLZ-extension is admissible for the inner problem. 

The criteria for being admissible does not include satisfying the cochain map property, but it can be shown that among all such admissible extensions, the one with minimal $L^2$-norm out of solutions to the inner problem satisfies the cochain property. 

In particular, the stability constant will not involve any trace inequality constants, and harmonic forms may be extended without issue as detailed in Remark~\ref{rem:ambient.topology}.
\end{remark}\medskip

\begin{figure}[t]
\centering
\includegraphics[width=\textwidth]{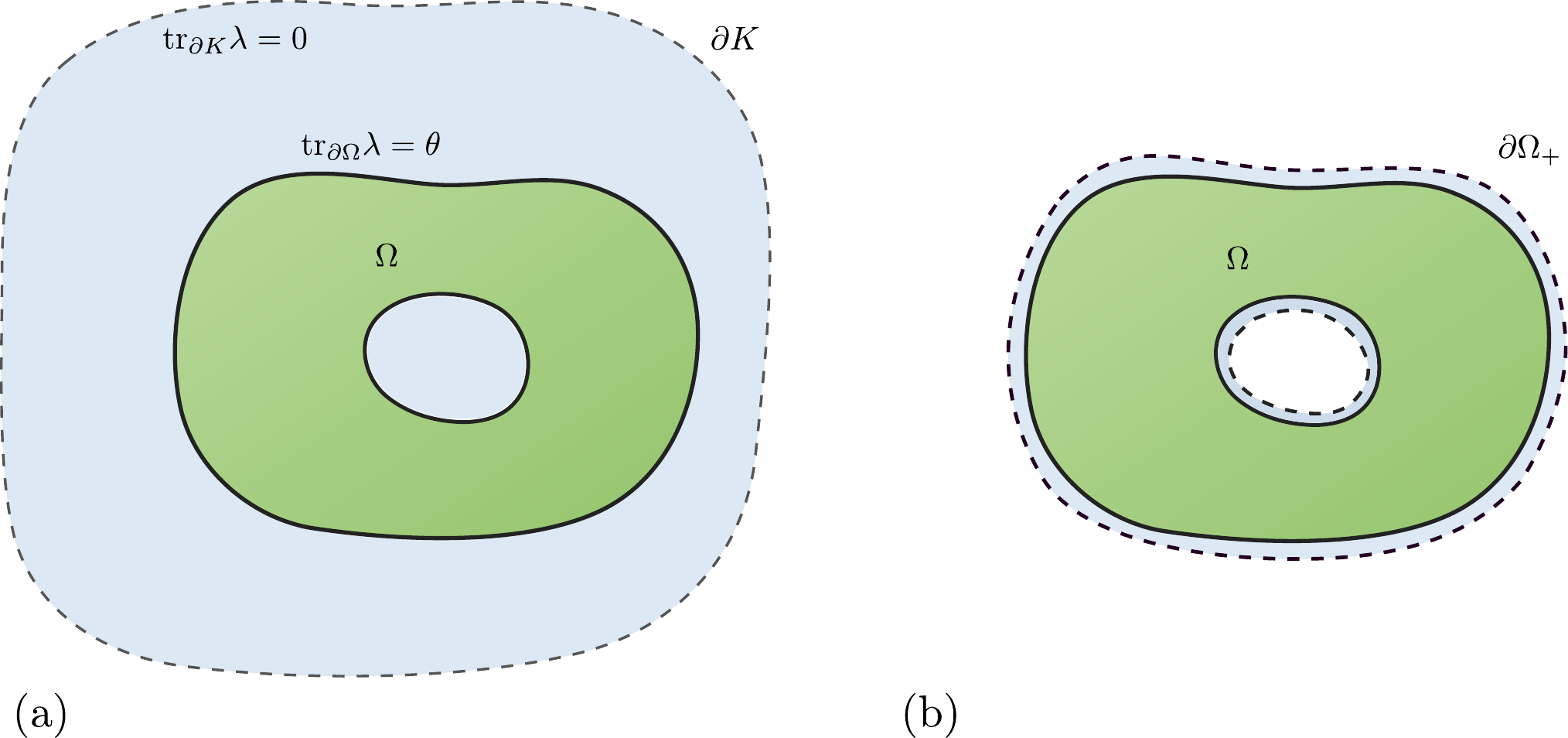}
\caption{Comparison between the present global extension setting and a collar-based construction. (a) A bounded Lipschitz domain $\Omega$ is contained in an ambient Lipschitz domain $K$, and the extension is constructed on the exterior region $A=K\setminus\overline{\Omega}$, allowing for arbitrary ambient topology. (b) In a collar-based setting, the ambient domain $\Omega_+\supset\Omega$ remains localised near $\Omega$ and has the same homotopy type.}
\label{fig:global..collar}
\end{figure}

\begin{remark}[Alternative trace liftings and domain regularity]%\label{rem:adm.comp}
The cutoff of the HLZ-extension is a convenient choice for constructing the exterior field in the recursive problem, but it is not the only possible one; see Figure~\ref{fig:global..collar}. An alternative is obtained by taking a cutoff of the pullback through a bi-Lipschitz retraction $R$ from a collar neighbourhood, denoted by $\Omega_+\supset\Omega$, as in \cite[Section 7]{Luukkainen1977ElementsOL}. Such a pullback satisfies a stability bound depending only on the Lipschitz constants of $R$ and $R^{-1}$.

However, this alternative lifting is inherently local: it is naturally defined on a collar neighbourhood $\Omega_+$ of $\Omega$, and therefore remains tied to the geometry and topology of $\Omega$. In particular, since $\Omega_+$ is obtained from a collar construction, it has the same homotopy type as $\Omega$. Thus, unlike the HLZ-extension, this approach does not directly provide an extension into an arbitrary ambient domain $K$ fixed independently of $\Omega$ and possibly of different topology.

Moreover, since the collar neighbourhood is localised near $\partial\Omega$, the cutoff must begin to decay immediately outside $\Omega$. The relevant distance parameter then becomes $\rho_+\coloneqq\dist(\partial\Omega,\partial\Omega_+)$, which may be very small. In particular, changing $\Omega$ also changes $\Omega_+$, and may therefore increase $\rho_+^{-1}$. By contrast, when $K$ is fixed, the parameter $\rho\coloneqq\dist(\partial\Omega,\partial K)$ has a more favourable monotonicity with respect to the size of the domain: if $\Omega'\subset\Omega$, then $\rho'>\rho$. For this reason, the cutoff of the HLZ-extension is better suited to the present objective.

More generally, the recursive construction only requires a functional setting in which the spaces $H\Lambda^k$ and the tangential trace operator are well defined on $\Omega$ and $K$. This suggests that the same strategy may be adapted to weakly Lipschitz domains in the sense of \cite{licht2019smoothed}. In that setting, the HLZ-extension is not directly available, but one may instead use a cutoff of a pullback extension to a weak collar neighbourhood; see, for instance, \cite[Lemma 7.3]{licht2019smoothed}. The stability considerations above then apply in the same way, with the corresponding collar parameter replacing $\rho$.
\end{remark}

\subsection{Equivalent gauge formulation and connection with previous literature}
\label{sec:gauge.formulation}
An alternative argument to prove Lemma~\ref{lem:well.posedness.cochain.extension} is to state in the inductive step that the corresponding minimisation problem to \eqref{eq:recursive.def.cont} is a minimum norm problem over a closed and convex set.

Indeed, the recursive definition of the extension admits a natural gauge interpretation, which connects our construction with previous gauge-based extension procedures, in particular the local extension operators of Falk and Winther \cite{falk2014local}.

Fix $\omega\in H\Lambda^k(\Omega)\cap\mathfrak H^k(\Omega)^\perp$, and set $\xi\coloneqq \cext{k+1}(d\omega)|_A\in L^2\Lambda^{k+1}(A)$. Consider the affine space
\[
\mathcal A_{\omega,\xi}^k
\coloneqq
\Bigl\{
\mu\in H\Lambda^k(A):
d\mu=\xi,\,
\tr_{\partial\Omega}\mu=\tr_{\partial\Omega}\omega,\,
\tr_{\partial K}\mu=0,\,
(\mu,q)_{L^2(A)}=c_q^k\ \forall q\in\mathfrak H_0^k(A)
\Bigr\}.
\]
By Lemma~\ref{lem:well.posedness.cochain.extension} the set is nonempty. 
The solution to system \eqref{eq:recursive.def.cont} is precisely $$\arg\min_{\mu\in \mathcal A_{\omega,\xi}^k} \|\mu\|^2_{L^2\Lambda^k(A)}.$$

\begin{proposition}[Gauge formulation of the recursive step]
\label{prop:gauge.equivalent.recursive}
A form $\lambda\in\mathcal A_{\omega,\xi}^k$ is the minimal-norm solution of \eqref{eq:recursive.def.cont} if and only if
$
(\lambda,d\tau)_{L^2\Lambda^k(A)}=0
$
for all $\tau\in H_0\Lambda^{k-1}(A).$
\end{proposition}
\begin{proof}
Any $\lambda\in \mathcal A_{\omega,\xi}^k$ is a minimiser of the $L^2$-norm over $A$ if and only if the first variation vanishes in every admissible direction. 
Let $\lambda\in\mathcal A_{\omega,\xi}^k$, and let $z$ be an admissible variation to the minimal norm problem. 
Then $dz=0$, $\tr_{\partial A}z=0$, and $(z,q)_{L^2(A)}=0$ for all $q\in\mathfrak H_0^k(A)$. By the relative Hodge decomposition on $A$, it follows that $z=d\tau$ for some $\tau\in H_0\Lambda^{k-1}(A)$. 
Therefore, the first variation condition reads as
\[
(\lambda,d\tau)_{L^2\Lambda^k(A)}=0
\qquad
\forall \tau\in H_0\Lambda^{k-1}(A).
\]
This proves the claim.
\end{proof}

\begin{remark}[Connection with Falk--Winther]
When the relative harmonic space $\mathfrak H_0^k(A)$ is trivial, the harmonic constraints in the definition of $\mathcal A_{\omega,\xi}^k$ disappear. Dropping also the zero trace condition on $\partial K$, Proposition~\ref{prop:gauge.equivalent.recursive} reduces to the standard gauge condition used by Falk and Winther \cite{falk2014local} in their construction of local extension operators on simplicial stars. 

Thus, in the topologically trivial case, our recursive construction reduces to a Falk--Winther-type gauge formulation. In the general case, the additional harmonic conditions are precisely what encodes the topological obstruction on the exterior domain $A$ and allows the construction to extend to arbitrary topologies.
\end{remark}

\section{Stability of unfitted methods}
\label{sec:cutfeec}
In this section, we leverage the cochain map property of the graded extension family to derive the uniform discrete Poincaré inequalities needed in the stability analysis of unfitted methods, with particular emphasis on the Hodge--Laplace problem.

\subsection{Uniform discrete Poincaré inequality}\label{sec:unif.discrete.Poincare}
Let $\{\Omega_h\}_{h>0}$ denote a family of polytopal domains approximating $\Omega$ from the exterior. We assume that this family has uniformly bounded Lipschitz character, and is uniformly bounded in the sense that there exists a fixed convex region $K$ such that $\Omega \subset \Omega_h \subset K$ for all $h>0$.

For each mesh underlying $\Omega_h$, we assume the existence of a bounded cochain projection $\Pi_h^k:H\Lambda^k(\Omega_h)\to\Lambda_h^k(\Omega_h)$ onto a discrete differential form space $\Lambda_h^k(\Omega_h)\subset H\Lambda^k(\Omega_h)$. The projection is assumed to satisfy the cochain property $d\circ\Pi_h^k=\Pi_h^{k+1}\circ d$ and to be uniformly bounded with respect to $h$, namely
\begin{equation}\label{eq:proj.unif.bound}
\|\Pi_h^k\tau\|_{H\Lambda^k(\Omega_h)} \le C_{\mathrm{proj}}\|\tau\|_{H\Lambda^k(\Omega_h)},
\end{equation} 
for a constant $C_{\mathrm{proj}}$ independent of $h$. Such projections exist, for example, on simplicial and cubical meshes; see \cite{arnold_feec_homological_2006,arnold_feec_hodge_2010}. In the lowest-order case, one may take the standard Whitney forms.

The following theorem shows how the extension operator from Theorem~\ref{thm:global.extension} yields a discrete Poincaré inequality with a constant independent of both the mesh size $h$ and the specific active domain $\Omega_h$. We use the equivalent formulation of the Poincaré inequality proved in \cite[Theorem~4, Corollary~5]{di2025uniformpoincare}.

\begin{theorem}[$h$-uniform discrete Poincaré inequality]\label{thm:discrete.uniform.poincare}
Let $k\in\{0,\dots,n-1\}$. For any polytopal domain $\Omega_h$ in the family and every discrete form $\omega_h\in\Lambda_h^k(\Omega_h)$, there exists a discrete potential $\tau_h\in\Lambda_h^k(\Omega_h)$ such that $d\tau_h=d\omega_h$ and
\begin{equation*}
\|\tau_h\|_{L^2\Lambda^k(\Omega_h)} \le C_{\mathrm{DP}}\|d\omega_h\|_{L^2\Lambda^{k+1}(\Omega_h)},
\end{equation*}
where the constant $C_{\mathrm{DP}}$ depends only on the set $K$, the extension constant $\tilde C$ from Theorem~\ref{thm:global.extension}~(b), and the projection bound $C_{\mathrm{proj}}$, but is independent of $h$ and of the specific active domain $\Omega_h$.
\end{theorem}

\begin{proof}
Let $\sigma_h\coloneqq d\omega_h\in\Lambda_h^{k+1}(\Omega_h)$. Since $\sigma_h$ is exact on $\Omega_h$, it is closed, that is, $d\sigma_h=0$, and $L^2$-orthogonal to the harmonic space $\mathfrak H^{k+1}(\Omega_h)$.

We apply the extension operator from Theorem~\ref{thm:global.extension} (b) to $\sigma_h$ and define $\tilde{\sigma}\coloneqq \cext{k+1}\sigma_h\in H\Lambda^{k+1}(K)$. Because $\sigma_h$ is orthogonal to harmonic forms, the extension operator commutes with the exterior derivative, and therefore
\[
d\tilde{\sigma}=d(\cext{k+1}\sigma_h)=\cext{k+2}(d\sigma_h)=0 \qquad \text{in }K.
\]

Since $K$ is convex, hence contractible, there exists a potential $\eta\in H\Lambda^k(K)$ such that $d\eta=\tilde{\sigma}$. Moreover, choosing $\eta\in(\ker d|_K)^\perp$, the continuous Poincaré inequality on $K$ gives
\begin{equation}\label{eq:poincare_continuous_K}
\|\eta\|_{H\Lambda^k(K)} \le C_P(K)\|\tilde{\sigma}\|_{L^2\Lambda^{k+1}(K)},
\end{equation}
where $C_P(K)$ denotes the Poincaré constant of $K$.

We then define the discrete potential by
\[
\tau_h\coloneqq \Pi_h^k(\eta|_{\Omega_h}).
\]

We first check that $\tau_h$ is indeed a discrete potential for $\sigma_h$. Since $\Pi_h^\bullet$ is a cochain projection,
\[
d\tau_h
= d\Pi_h^k(\eta|_{\Omega_h})
= \Pi_h^{k+1}(d\eta|_{\Omega_h})
= \Pi_h^{k+1}(\tilde{\sigma}|_{\Omega_h})
\overset{\text{Thm.~\ref{thm:global.extension}(i)}}{=}
\Pi_h^{k+1}\sigma_h
= \sigma_h,
\]
because $\Pi_h^{k+1}$ is a projection onto $\Lambda_h^{k+1}(\Omega_h)$. Hence $d\tau_h=\sigma_h=d\omega_h$.

It remains to prove the uniform bound. Using the boundedness of the projection \eqref{eq:proj.unif.bound}, the continuous Poincaré estimate \eqref{eq:poincare_continuous_K}, and the stability of the extension operator from Theorem~\ref{thm:global.extension} (ii) with the considerations of Remark~\ref{rem:exact.extension.improved.bound}, we obtain
\begin{align*} 
\|\tau_h\|_{L^2\Lambda^k(\Omega_h)}
&= \|\Pi_h^k(\eta|_{\Omega_h})\|_{L^2\Lambda^k(\Omega_h)}\\
&\le C_{\mathrm{proj}}\,\|\eta\|_{H\Lambda^k(\Omega_h)}\\
&\le C_{\mathrm{proj}}\,\|\eta\|_{H\Lambda^k(K)}\\
&\le C_{\mathrm{proj}}C_P(K)\,\|\tilde{\sigma}\|_{L^2\Lambda^{k+1}(K)}\\
&\le C_{\mathrm{proj}}C_P(K)(1+\tilde C)\,\|\sigma_h\|_{L^2\Lambda^{k+1}(\Omega_h)}.
\end{align*}
This proves the claim with $C_{\mathrm{DP}}=C_{\mathrm{proj}}C_P(K)(1+\tilde C)$.
\end{proof}

Note that one may also use the extension from Theorem~\ref{thm:global.extension}~(a), at the cost of an additional factor $\dist(\partial\Omega,\partial K)^{-1}$ in the constant.

\subsection{Uniform ghost norm Poincaré inequality}
To prove the uniform inf-sup condition for a CutFEEC discretisation of the Hodge--Laplace equation \cite{erik_cutfeec}, the key missing ingredient is a Poincaré inequality in the stabilised norm $\|\cdot\|_s$ for forms orthogonal to the kernel.

Let $\Omega\subset\Real^n$ be a physical domain with Lipschitz boundary.
In the CutFEM framework \cite{burman2025cut}, on which CutFEEC is based, the domain is embedded into a polytopal background domain $\Omega_0$ equipped with a background mesh $\mathcal{T}_{0,h}$.
The active mesh is defined as
\[
\mathcal{T}_h\coloneqq\{T\in\mathcal{T}_{0,h}:T\cap\Omega\neq\emptyset\},
\]
and the union of its elements forms the active domain $\Omega_h$, which strictly contains the physical domain, that is, $\Omega\subset\Omega_h$.
The family of active domains therefore fits into the framework of Section~\ref{sec:unif.discrete.Poincare}, and moreover satisfies $\Omega_{h'}\subset\Omega_h$ whenever $h'<h$.

To deal with the ill-conditioning caused by elements having arbitrarily small intersections with $\Omega$, one introduces a \emph{ghost penalty stabilisation} form $s(\cdot,\cdot)$; see \cite[Eq. (5.1)]{erik_cutfeec}.
This leads to the stabilised inner product
\[
(\omega,\zeta)_s\coloneqq(\omega,\zeta)_\Omega+s(\omega,\zeta)\qquad\forall \omega,\zeta\in\Lambda_h^k(\Omega_h),
\]
and to the induced norm $\|\cdot\|_s$.
A fundamental property of this stabilisation is that the norm $\|\cdot\|_s$ is uniformly equivalent to the standard $L^2$-norm on the active domain, namely $\|\cdot\|_{L^2\Lambda^k(\Omega_h)}$, independently of the mesh size and of the cut configuration.

The stabilised inner product naturally induces a modified space of discrete harmonic forms,
\[
\mathfrak H_s^k\coloneqq\{\rho_h\in\Lambda_h^k(\Omega_h):d\rho_h=0,\ (\rho_h,d\tau_h)_s=0\quad\forall \tau_h\in\Lambda_h^{k-1}(\Omega_h)\},
\]
which replaces the standard $L^2(\Omega_h)$-orthogonal harmonic space.
Accordingly, the discrete spaces admit the stabilised Hodge decomposition
\[
\Lambda_h^k(\Omega_h)=(\ker d_h)^{\perp_s}\oplus_s d\Lambda_h^{k-1}(\Omega_h)\oplus_s \mathfrak H_s^k,
\]
where $\oplus_s$ denotes orthogonality with respect to the stabilised inner product.
This decomposition is the key structural ingredient in the analysis of the unfitted Hodge--Laplace problem.

\begin{corollary}[$h$-uniform stabilised Poincaré inequality]\label{cor:stabilised_poincare}
For any $\eta_h\in(\ker d_h)^{\perp_s}$, there exists a constant $C_S>0$, independent of $h$ and of the active domain $\Omega_h$, such that
\[
\|\eta_h\|_s \leq C_S\|d\eta_h\|_s.
\]
\end{corollary}

\begin{proof}
Let $\eta_h\in(\ker d_h)^{\perp_s}$. By the uniform equivalence between the $L^2(\Omega_h)$-norm and the stabilised norm proved in \cite{erik_cutfeec}, we have
\[
\|\eta_h\|_s \approx \|\eta_h\|_{L^2\Lambda^k(\Omega_h)}
\qquad\text{and}\qquad
\|d\eta_h\|_s \approx \|d\eta_h\|_{L^2\Lambda^{k+1}(\Omega_h)}.
\]
Since $d\eta_h\in d\Lambda_h^k(\Omega_h)$, Theorem~\ref{thm:discrete.uniform.poincare} yields a discrete potential $\tau_h\in\Lambda_h^k(\Omega_h)$ such that $d\tau_h=d\eta_h$ and
\[
\|\tau_h\|_{L^2\Lambda^k(\Omega_h)} \leq C_{\mathrm{DP}}\|d\eta_h\|_{L^2\Lambda^{k+1}(\Omega_h)},
\]
where $C_{\mathrm{DP}}$ is independent of both $h$ and $\Omega_h$. 
Because $\eta_h$ is orthogonal to $\ker d_h$ with respect to $(\cdot,\cdot)_s$, it is the unique minimiser of the norm $\|\cdot\|_s$ among all discrete forms having the same exterior derivative $d\eta_h$. Hence
\[
\|\eta_h\|_s\le \|\tau_h\|_s.
\]
Combining this minimality property with the two norm equivalences above and the estimate for $\tau_h$ proves the result, with $C_S$ absorbing the corresponding constants.
\end{proof}

With Corollary~\ref{cor:stabilised_poincare} in hand, the uniform well-posedness of the unfitted scheme for the Hodge--Laplace equation in \cite[Eq. (6.1)]{erik_cutfeec} follows from the standard Babuška--Brezzi theory for mixed methods.

\section{Additional applications}
\label{sec:applications}
In this section, we present two further applications of the cochain extension operator in the continuous setting: a uniform Poincaré inequality for differential forms and a lower bound for the first Neumann eigenvalue on non-convex domains.

\subsection{Uniform Poincaré inequalities}
As before, let $\Omega\subset \Real^n$ be a bounded domain with Lipschitz boundary.
The dependence of the Poincaré constant on the domain geometry in the classical Poincaré inequality for functions has long been of interest; see, for instance, \cite{bebendorf2003poincareconvex}.
For convex domains in $\Real^n$, the optimal constant is bounded above by $\mathrm{diam}(\Omega)/\pi$; see \cite{payne1957lower,payne1960optimalpoincare}.
The same type of bound also extends to Poincaré inequalities for differential forms.
For general $k$-forms on convex domains, Guerini \cite{guerini2004poincareDiameterConvex} obtained the sharper upper bound $\mathrm{diam}(\Omega)/C_1$, where the constant $C_1$ depends only on the dimension and the form degree.
A natural question is therefore whether one can establish a Poincaré inequality with a constant free from restrictive dependencies on the domain geometry.
Under suitable geometric and topological assumptions, several affirmative results are known \cite{ruiz2012uniformity,boulkhemair2007uniform,thomas2014uniform,valette2021uniform}.

We recall the result of Ruiz \cite{ruiz2012uniformity}, which applies to the most general class of domains among these works.
We also note that the result of \cite{thomas2014uniform} requires connectedness of the domains in the family, since the proof of \cite[Theorem 1]{thomas2014uniform} relies on \cite{boulkhemair2007uniform}.
For a family of connected domains with uniformly bounded diameters and satisfying a uniform cone property, which in particular holds for uniformly Lipschitz families, the following theorem holds.

\begin{theorem*}[Ruiz \cite{ruiz2012uniformity}]
There is a constant $C_R$ such that for all $\Omega$ in the family and for all $u\in H^1(\Omega)=H\Lambda^0(\Omega)$ it holds
\begin{equation*}
\|u-\overline{u}\|_{L^p(\Omega)} \leq C_R \|\nabla u\|_{L^p(\Omega)},
\end{equation*}
where $\overline{u}$ is the average of $u$ over $\Omega$.
\end{theorem*}

While this result is powerful, it is restricted to scalar functions, that is, to form degree $k=0$.
We now use our cochain extension to obtain a substantial generalisation to differential forms of arbitrary degree, in the case $p=2$.

Let $D>0$ and $L>0$ be fixed constants, and define the uniformly bounded family of Lipschitz domains
\begin{equation}\label{eq:domain.family}
\mathcal{F}\coloneqq\{\Omega\subset\Real^n:\Omega\text{ is Lipschitz},\ \mathrm{diam}(\Omega)\le D,\ \text{and the Lipschitz character of }\Omega\text{ is bounded by }L\}.
\end{equation}
The following argument mirrors the proof of the discrete uniform Poincaré inequality in Theorem~\ref{thm:discrete.uniform.poincare}, with the continuous restriction replacing the discrete projection step.

\begin{theorem}[Uniform Poincaré inequality]\label{thm:uniform.poincare}
Let $k \in \{0,\dots,n-1\}$.
For any domain $\Omega\in\mathcal{F}$ and every form $\omega \in H\Lambda^{k}(\Omega)$, there exists a form $\tau \in H\Lambda^k(\Omega)$ such that $d\tau = d\omega$ and
\begin{equation*}
\|\tau\|_{L^2\Lambda^k(\Omega)} \leq C_U \|d\omega\|_{L^2\Lambda^{k+1}(\Omega)},
\end{equation*}
where the constant $C_U$ depends only on the diameter bound $D$, the uniform Lipschitz character $L$, and the form degree $k$.
\end{theorem}
\begin{proof}
Take any $\Omega\in\mathcal{F}$. Since $\mathrm{diam}(\Omega)\le D$, we can enclose $\Omega$ in a cube $K$ such that $\mathrm{diam}(K)\le cD$ for some uniform constant $c$. 
Set $\sigma\coloneqq d\omega\in L^2\Lambda^{k+1}(\Omega)$.
Since $\sigma$ is exact, we have $\sigma\in H\Lambda^{k+1}(\Omega)\cap\mathfrak H^{k+1}(\Omega)^\perp$.
We may therefore extend $\sigma$ to $K$ using Theorem~\ref{thm:global.extension}~(b), and define $\tilde{\sigma}\coloneqq \cext{k+1}\sigma$.
Since $\sigma$ is exact, $\tilde{\sigma}$ is closed on $K$.

Applying the continuous Poincaré lemma on $K$, we obtain a potential $\eta \in H\Lambda^k(K)$ such that $d\eta = \tilde \sigma$ and $\|\eta\|_{L^2\Lambda^k(K)} \leq C_P(K)\|\tilde \sigma\|_{L^2\Lambda^{k+1}(K)}$.
We then define the local potential by restriction, namely $\tau\coloneqq\eta|_\Omega$.
Consequently, $d\tau=\tilde \sigma|_\Omega=\sigma=d\omega$.
The stability bound now follows exactly as in the final part of the proof of Theorem~\ref{thm:discrete.uniform.poincare}, with the discrete projection step omitted.
This yields $C_U\coloneqq C_P(K)(1+\tilde C)$.
\end{proof}

This result improves upon the $L^2$ case of Ruiz's theorem in three ways.
First, it applies to differential forms of all degrees $k$, not only to the scalar case $k=0$.
Second, it removes the need for topological assumptions such as connectedness: $\Omega$ may have arbitrarily many connected components, provided it remains a bounded Lipschitz domain.
Third, the uniform constant $C_U$ is independent of topology, depending only on the uniform bounds on the Lipschitz character and on the diameter of the ambient convex set used in the construction.

We also state the corresponding relative, or Dirichlet, version, which follows from the same argument using the zero-trace extension of Theorem~\ref{thm:global.extension}~(a).
\begin{corollary}[Uniform relative Poincaré inequality]\label{cor:uniform.relative.poincare}
Let $k \in \{0,\dots,n-1\}$. For any domain $\Omega\in\mathcal F$ and every form
$\omega \in H_0\Lambda^{k}(\Omega)$, there exists a form $\tau \in H_0\Lambda^k(\Omega)$ such that
$d\tau = d\omega$ and
\[
\|\tau\|_{L^2\Lambda^k(\Omega)} \leq C_{U,0} \|d\omega\|_{L^2\Lambda^{k+1}(\Omega)},
\]
where the constant $C_{U,0}$ depends only on the diameter bound $D$, the uniform Lipschitz character $L$, and the form degree $k$.
\end{corollary}
\begin{proof}
The proof is identical to that of Theorem~\ref{thm:uniform.poincare}, except that one uses Theorem~\ref{thm:global.extension} (a) instead of part~(b). Thus, for $\sigma\coloneqq d\omega$, one first constructs a compactly supported extension $\tilde\sigma\in H_0\Lambda^{k+1}(K)$ with $d\tilde\sigma=0$. Since $K$ is convex, the relative Poincaré lemma on $K$ provides $\eta\in H_0\Lambda^k(K)$ such that $d\eta=\tilde\sigma$ and
\[
\|\eta\|_{L^2\Lambda^k(K)} \le C_{P,0}(K)\|\tilde\sigma\|_{L^2\Lambda^{k+1}(K)}.
\]
Setting $\tau\coloneqq \eta|_\Omega$ yields $\tau\in H_0\Lambda^k(\Omega)$ and $d\tau=d\omega$, while the stability bound follows exactly as in the proof of Theorem~\ref{thm:uniform.poincare}. This gives $C_{U,0}\coloneqq C_{P,0}(K)(1+\tilde C)$.
\end{proof}

\begin{remark}[Uniform lifting of \cite{hiptmair2012universalextension}]\label{rem:lifting.lemma}
The formulation of the Poincaré inequalities in Theorems~\ref{thm:uniform.poincare} and \ref{thm:discrete.uniform.poincare} is sometimes referred to as a lifting lemma, since it amounts to bounding a right inverse of the exterior derivative.

In \cite[Corollary 5.4]{hiptmair2012universalextension}, such a lifting lemma is claimed without the use of a cochain extension. In our proofs, however, the cochain map property is a key ingredient, and indeed the proof of \cite[Corollary 5.4]{hiptmair2012universalextension} looks to be containing an error.
Their argument relies on a lifted regular decomposition $R+dN=\mathrm{Id}$ for operators $R:\Lambda^k(\Omega)\to\Lambda^k(\Omega)$ and $N:\Lambda^k(\Omega)\to\Lambda^{k-1}(\Omega)$ that increase the regularity by one order.
As written, for a datum $\omega\in\Lambda^k(\Omega)$ the potential $\eta\in\Lambda^{k-1}(\Omega)$ is taken to be $R\omega\in\Lambda^k(\Omega)$, which is impossible because these objects have different form degrees.

If the intended meaning was instead to take $\eta=N\omega$, then the proof still fails, and the obstruction is precisely the lack of a global cochain property.
Indeed, the definition $R\omega\coloneqq L(d\EHLZ\omega)|_\Omega$ involves a lifting $L:\Lambda^k(\Real^n)\cap\ker d\to\Lambda^{k-1}(\Real^n)$ of closed forms to compactly supported forms.
From the regular decomposition one has $dN\omega=\omega-R\omega$, and therefore $R\omega=0$ if and only if the extension $\EHLZ$ maps closed forms to closed forms.
\end{remark}

\subsection{Uniform lower bound for first Neumann eigenvalue of the Hodge Laplacian on Lipschitz domains}
We consider the coclosed Neumann eigenvalue problem for the Hodge Laplacian on a bounded Lipschitz domain $\Omega\subset \mathbb R^n$. 
In the literature this problem is also called the absolute eigenvalue problem \cite{chakradhar2025lower}. 
By the Hodge decomposition \eqref{eq:hodge.decomposition}, every form in $(\ker(d|_\Omega))^\perp$ is coclosed and satisfies the absolute boundary condition $\ntr_{\partial\Omega}\omega=0$ on $\partial\Omega$. This identifies the natural space in which the eigenvalue problem below is posed.

The Neumann eigenvalue problem for the Hodge Laplacian \cite{chakradhar2025lower} consists in finding $(\omega,\lambda)\in H\Lambda^k(\Omega)\times\mathbb R$ such that
\[
\left\{
\begin{aligned}
(d\delta+\delta d)\omega &= \lambda\omega &&\text{in } \Omega,\\
\ntr_{\partial\Omega}\omega &= 0 &&\text{on } \partial\Omega,\\
\ntr_{\partial\Omega} d\omega &= 0 &&\text{on } \partial\Omega.
\end{aligned}
\right.
\]
where $\ntr_{\partial\Omega}$ is the interior product with the outward unit normal vector field $\nu$ on $\partial\Omega$ for sufficiently regular forms. The coclosed Neumann eigenvalue problem for the Hodge Laplacian is obtained by adding the constraint $\delta\omega=0$ in $\Omega$. Since, to our knowledge, it is unknown whether the first coclosed Neumann eigenvalue is monotone with respect to the form degree on non-convex domains, it is of interest to establish lower bounds.

Applying Theorem~\ref{thm:uniform.poincare}, we obtain for every nonzero $\omega\in H\Lambda^k(\Omega)\cap(\ker(d|_\Omega))^\perp$ that
\[
\|\omega\|^2_{L^2\Lambda^k(\Omega)} \le C_U^2 \|d\omega\|^2_{L^2\Lambda^{k+1}(\Omega)}.
\]
Rewriting this in terms of the Rayleigh quotient, we obtain
\[
\frac{1}{C_U^2}\leq \inf_{\omega \in (\ker(d|_\Omega))^\perp \cap H\Lambda^k(\Omega)\setminus\{0\}}
\frac{\|d\omega\|^2_{L^2\Lambda^{k+1}(\Omega)}}{\|\omega\|^2_{L^2\Lambda^k(\Omega)}}=:\lambda_1\le \lambda_2 \leq \dots \nearrow \infty.
\]
Note that the boundary condition $\ntr_{\partial\Omega} d\omega=0$ is natural in the weak formulation and therefore does not appear explicitly in the variational characterisation.

Since the Poincaré constant and the first nonzero Neumann eigenvalue are reciprocals of one another, the following corollary is simply a restatement of Theorem~\ref{thm:uniform.poincare} for the same family $\mathcal F$ of domains defined by \eqref{eq:domain.family}. 

We also record a more explicit lower bound for non-convex domains, which generalises, in terms of the quartic power of the reciprocal of the diameter of the convex hull, the lower bound for planar annuli proved by Savo and Colbois \cite[Theorem~1]{colbois2021lower}. 
In this application, one wants the ambient convex set $K$ to be as close to $\Omega$ as possible. A natural way to achieve this is to start from a Lipschitz collar neighbourhood $\Omega_+\supset\Omega$ and then take its convex hull. See Figure~\ref{fig:convhull} for an illustration of such a set $\Omega_+$ and of $\mathrm{Conv}(\Omega_+)$.

\begin{corollary}[Uniform lower bound of first Neumann eigenvalue]\label{cor:uniform.neumann}
Let $k \in \{0,\dots,n-1\}$. For any $\Omega\in\mc F$, the first Neumann eigenvalue of the coclosed Hodge Laplacian on $\Omega$ is bounded below by $1/C_U^2$, where $C_U$ is the uniform Poincaré constant from Theorem~\ref{thm:uniform.poincare}.

In particular, given any bounded Lipschitz domain $\Omega$, we have the lower bound
\begin{equation}\label{eq:diam.conv.hull.ineq}
\lambda_1 \geq \frac{C}{\diam(\mathrm{Conv}(\Omega_+))^4},
\end{equation}
where $\Omega_+\supset\Omega$ is a Lipschitz collar neighbourhood and $C$ depends only on the Lipschitz character of $\Omega$, its dimension, and the form degree.
\end{corollary}

\begin{proof}
Let $K\coloneqq \mathrm{Conv}(\Omega_+)$. Since $K$ is convex, Guerini \cite{guerini2004poincareDiameterConvex} gives $C_P(K)\le \diam(K)/C_1 = \diam(\mathrm{Conv}(\Omega_+))/C_1$. Applying Theorem~\ref{thm:global.extension}~(b) with this choice of $K$ yields an extension constant $1+\tilde C$ depending on the Lipschitz character of $\Omega$, the dimension, and the geometry of $K$.
Hence, the proof of Theorem~\ref{thm:uniform.poincare} with this choice of convex ambient set $K$, gives
\[
C_U = C_P(K)(1+\tilde C)\lesssim \diam(\mathrm{Conv}(\Omega_+))^2,
\]
with an implicit constant depending only on the Lipschitz character of $\Omega$, the dimension, and the form degree. Since $\lambda_1\ge C_U^{-2}$, this proves \eqref{eq:diam.conv.hull.ineq}.
\end{proof}

\begin{figure}[t]
\centering
\includegraphics[width=0.4\textwidth]{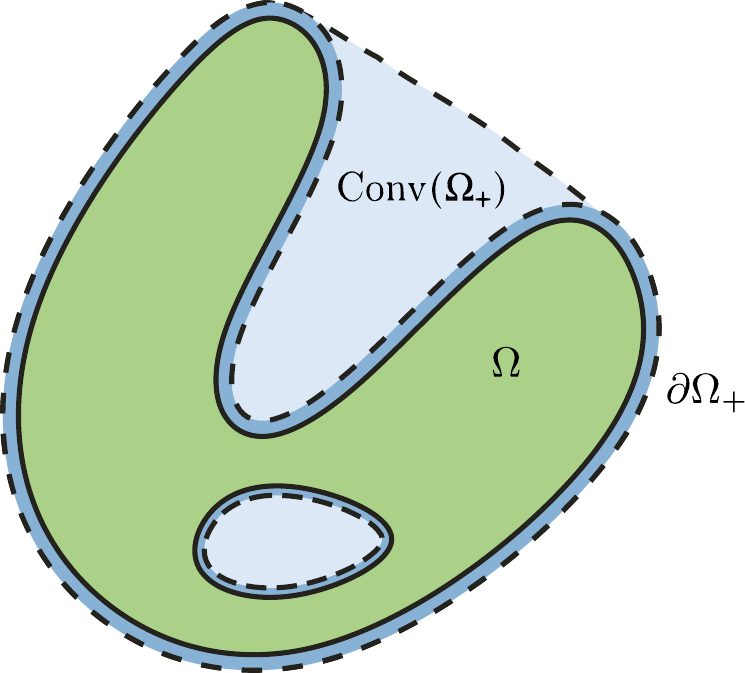}
\caption{Illustration of the convex hull of a collar neighbourhood $\Omega_+$.}
\label{fig:convhull}
\end{figure}

\begin{remark}
We remark that if $A=K\setminus\overline{\Omega}$ is itself assumed Lipschitz, then we can take $\Omega_+ = \Omega$ in the above Corollary~\ref{cor:uniform.neumann}. 
Also, since we use Theorem~\ref{thm:global.extension}~(b), there is no reciprocal of the distance between the boundaries of $A=\mathrm{Conv}(\Omega_+)\setminus\overline{\Omega}$ in the constant $C$. To be more precise,
\begin{align*}
C_U &\lesssim C_P(\mathrm{Conv}(\Omega_+))\bigl(2+C_{P,0}(\mathrm{Conv}(\Omega_+))\CHLZ + C_{\mathrm{tr}}(A)\,C_{\mathrm{tr}}(\Omega)\bigr) \\
&\lesssim \diam(\mathrm{Conv}(\Omega_+))^2(\CHLZ + C_{\mathrm{tr}}(A)\,C_{\mathrm{tr}}(\Omega)) =: \diam(\mathrm{Conv}(\Omega_+))^2/C.
\end{align*}
\end{remark}

\section{Further discussion and open questions}\label{sec:open.questions}
We conclude by discussing several open problems and possible future directions suggested by the construction and applications developed in this paper.

Our graded extension family is only bounded in $H\Lambda^k(K)$. Even when $\Omega$ and $K$ are smooth, elliptic regularity applied to the first-order potential problem, equivalently to the gauge formulation in Section~\ref{sec:gauge.formulation}, yields improved regularity only on each side of the interface $\partial\Omega$. Recovering higher regularity across $\partial\Omega$ would require additional compatibility conditions. This leads naturally to the question of whether Theorem~\ref{thm:global.extension} admits a strict higher-regularity analogue. 

Let $m\in \mathbb N_0$. For a bounded Lipschitz domain $\Omega\subset \mathbb R^n$ and form degree $k$, define $H^m\Lambda^k(\Omega)$ as the Sobolev spaces of $L^2$-integrable differential forms with coefficients in $H^m(\Omega)$, with convention $H^0(\Omega)=L^2(\Omega).$ Then define
\begin{align}\label{def:Hmm.space}
H^{(m,m)}\Lambda^k(\Omega)\coloneqq
\bigl\{\omega\in H^m\Lambda^k(\Omega): d\omega\in H^m\Lambda^{k+1}(\Omega)\bigr\},
\end{align}
equipped with the graph norm
\[
\|\omega\|_{H^{(m,m)}\Lambda^k(\Omega)}^2
\coloneqq
\|\omega\|_{H^m\Lambda^k(\Omega)}^2
+
\|d\omega\|_{H^m\Lambda^{k+1}(\Omega)}^2.
\]
See \cite[Section 2]{hiptmair2012universalextension} for more details.

\begin{openquestion}
For simplicity, let us first focus on the variant in Theorem~\ref{thm:global.extension}~(b). 
Can one construct, for each form degree $k$, a bounded operator
\[
E_m^k:
H^{(m,m)}\Lambda^k(\Omega)\cap \mathfrak H^k(\Omega)^\perp
\to
H^{(m,m)}\Lambda^k(K)
\]
such that, for every
$\omega\in H^{(m,m)}\Lambda^k(\Omega)\cap \mathfrak H^k(\Omega)^\perp$,
\begin{enumerate}[label=(\roman*)]
\item $E_m^k\omega|_\Omega=\omega$,
\item $\|E_m^k\omega\|_{H^{(m,m)}\Lambda^k(K)}
\le C\,\|\omega\|_{H^{(m,m)}\Lambda^k(\Omega)}$,
\item $d\,E_m^k\omega = E_m^{k+1}(d\omega)$ in $K$?
\end{enumerate}

If, in addition, $\Omega$ is strictly contained in $K$, can one impose vanishing tangential traces for all directions on $\partial K$ so that the corresponding extension by zero defines an operator into $H^{(m,m)}\Lambda^k(\mathbb R^n)$?
\end{openquestion}

The lower bound in Corollary~\ref{cor:uniform.neumann} is not asymptotically optimal in the power of the diameter. This loss comes from the fact that the argument uses the Poincaré constant twice. At present, it is not clear how to remove the second occurrence, which appears in the proof of Theorem~\ref{thm:uniform.poincare}. An alternative route would be to combine our construction with a whole-space lifting such as \cite[Lemma~5.1]{hiptmair2012universalextension}, but the low-frequency part of the estimate seems in general to introduce a dependence on $\diam(K)^{n/2}$. Another possibility would be to prove a direct $L^2$-stability estimate for the extension of closed forms, avoiding a second use of the Poincaré inequality.

\begin{openquestion}
Can inequality \eqref{eq:diam.conv.hull.ineq} be improved so as to scale quadratically with the inverse of the diameter of the convex hull of $\Omega$?
\end{openquestion}

A natural generalisation of our work is to ask whether one can construct a uniformly $L^p$-bounded graded extension family with a corresponding cochain property. Such a result would lead to uniform $L^p$-Poincaré inequalities extending those of Ruiz \cite{ruiz2012uniformity} and Boulkhemair \cite{boulkhemair2007uniform} to differential forms of arbitrary degree. For $p\neq 2$, however, the formulation can no longer rely on $L^2$-orthogonality to harmonic forms, so the topological obstruction would have to be encoded in a different way.

\begin{openquestion}%\label{q:2}
Is there a uniformly $L^p$-bounded graded extension family on bounded Lipschitz domains that satisfies a suitable cochain property, modulo the natural de Rham cohomological obstruction?

More precisely, can one formulate and prove an $L^p$ analogue of Theorem~\ref{thm:global.extension} for $1<p<\infty$, with stability bounds uniform on natural families of Lipschitz domains and with a formulation that remains valid for domains of arbitrary topology?

At present, it is not clear what the correct replacement is, in the $L^p$ setting, for the $L^2$-orthogonality condition to harmonic fields used in this paper, nor how such a formulation should be encoded in a variational construction.
\end{openquestion}

\section*{Acknowledgements}
The authors thank professor Jérôme Droniou for many useful comments and suggestions.

Erik Nilsson acknowledges the funding of the European Union via the ERC Synergy, NEMESIS, project number 101115663.
Silvano Pitassi acknowledges the funding of the European Union via the MSCA EffECT, project number 101146324.
Views and opinions expressed are however those of the authors only and do not necessarily reflect those of the European Union or the European Research Council Executive Agency.
Neither the European Union nor the granting authority can be held responsible for them.

\appendix
% \numberwithin{theorem}{section}
% \renewcommand{\thetheorem}{\thesection\arabic{theorem}}
% \setcounter{theorem}{0}
\section{Poincaré constants and boundary conditions}\label{appendix:poincare}
We collect here some basic facts from the literature regarding Poincaré constants. 
The Dirichlet eigenvalue problem for the Hodge Laplacian, also called the relative eigenvalue problem \cite{chakradhar2025lower}, consists in finding $(\omega,\kappa)\in H^*\Lambda^k(\Omega)\times\mathbb R$ such that
\[
\left\{
\begin{aligned}
(d\delta+\delta d)\omega &= \kappa\omega &&\text{in } \Omega,\\
\tr_{\partial\Omega}\omega &= 0 &&\text{on } \partial\Omega,\\
\tr_{\partial\Omega} \delta\omega &= 0 &&\text{on } \partial\Omega.
\end{aligned}
\right.
\]
It is dual to the Neumann eigenvalue problem via the Hodge star operator, and the first nonzero eigenvalues are related by the equation $\lambda_1^{(k)}=\kappa_1^{(n-k)}$, where $k$ is the form degree \cite{savo2011hodge}.

Let $\Omega$ be convex. 
We note the following relations; in order of appearance, see \cite{payne1960optimalpoincare}, \cite[Proposition 4.6]{seto2019sharpDirichlet} and \cite[Theorem 2.6]{guerini2004eigenvaluedegreemonotone}:
\begin{itemize}
  \item $\pi^2/\diam(\Omega)^2 \le\lambda_1^{(0)} = \kappa_1^{(n)}$;
  \item $n\pi^2/\diam(\Omega)^2 \le\kappa_1^{(0)} = \lambda_1^{(n)}$;
  \item $\lambda_1^{(0)}= \lambda_1^{(2)}\le \dots \le \lambda_1^{(n)}$;
  \item $\kappa_1^{(n)}= \kappa_1^{(n-1)}\le \dots \le \kappa_1^{(0)}$;
  \item $C^{(k)}_P(\Omega) = 1/\sqrt{\lambda_1^{(k)}}$;
  \item $C^{(k)}_{P,0}(\Omega) = 1/\sqrt{\kappa_1^{(k)}}$.
\end{itemize}
In summary, if $\Omega$ is convex, both Poincaré constants satisfy the upper bound 
\[
C_P,\, C_{P,0} \le \diam(\Omega)/\pi
\]
for all form degrees.
Moreover, for $k=0$, the Dirichlet Poincaré constant satisfies the tighter bound $C^{(0)}_{P,0}\le \diam(\Omega)/{(\sqrt{n}\pi)}$.

It is also well known that the Dirichlet Poincaré constant is monotone with respect to domain inclusion. For completeness, we provide a proof adapted to our context, where $A=K\setminus\overline{\Omega}$. 

\begin{lemma}\label{lem:dirichlet.poincare.monotone}
The Dirichlet Poincaré constants satisfy
\[
C_{P,0}(A)\le C_{P,0}(K).
\]
\end{lemma}
\begin{proof}
Let $u\in H_0\Lambda^k(A)\cap (\ker d|_A)^\perp$, and let $\tilde u$ be its extension by zero to $K$.
Since $u$ has vanishing tangential trace on $\partial A$, the extension $\tilde u$ belongs to $H_0\Lambda^k(K)$.
Moreover, $\tilde u\in (\ker d|_K)^\perp$. Indeed, for any $q\in \ker d|_K$, $(\tilde u,q)_{L^2\Lambda^k(K)}=(u,q|_A)_{L^2\Lambda^k(A)}=0$, since $q|_A\in\ker d|_A$ and $u\in (\ker d|_A)^\perp$.

Applying the Dirichlet Poincaré inequality on $K$ to $\tilde u$, we obtain
\[
\|u\|_{L^2\Lambda^k(A)} = \|\tilde u\|_{L^2\Lambda^k(K)} \le C_{P,0}(K)\|d\tilde u\|_{L^2\Lambda^{k+1}(K)} = C_{P,0}(K)\|du\|_{L^2\Lambda^{k+1}(A)}.
\]
Since this holds for every $u\in H_0\Lambda^k(A)\cap (\ker d|_A)^\perp$, it follows from the definition of $C_{P,0}(A)$ as the optimal Dirichlet Poincaré constant on $A$ that $C_{P,0}(A)\le C_{P,0}(K)$.
\end{proof}

\section{Weak trace estimates for differential forms}\label{appendix:trace.scaling}
Throughout this appendix, let $\Omega\subset\mathbb R^n$ be a bounded Lipschitz domain. As in Section~\ref{sec:notations}, all spaces and operators associated with form degrees outside $\{0,\dots,n\}$ are understood to be trivial.

For $k\in\{0,\dots,n\}$, we recall the strong tangential trace space
\[
H^{1/2}\Lambda^k(\partial\Omega):=\tr_{\partial\Omega}\bigl(H^1\Lambda^k(\Omega)\bigr),
\]
endowed with the quotient norm $\|\hat\omega\|_{H^{1/2}\Lambda^k(\partial\Omega)}:=\inf\Bigl\{\|u\|_{H^1\Lambda^k(\Omega)}:\ \tr_{\partial\Omega}u=\hat\omega\Bigr\}$.
We also introduce the strong normal trace space
\[
N^{1/2}\Lambda^k(\partial\Omega):=\ntr_{\partial\Omega}\bigl(H^1\Lambda^{k+1}(\Omega)\bigr),
\]
endowed with the quotient norm $\|\hat\psi\|_{N^{1/2}\Lambda^k(\partial\Omega)}:=\inf\Bigl\{\|\eta\|_{H^1\Lambda^{k+1}(\Omega)}:\ \ntr_{\partial\Omega}\eta=\hat\psi\Bigr\}$.
We equip the weak trace spaces $T^{-1/2}\Lambda^k(\partial\Omega)$ and $N^{-1/2}\Lambda^{k-1}(\partial\Omega)$ with the dual norms induced by the boundary pairing \eqref{eq:weak.bdry.paring} against $N^{1/2}\Lambda^k(\partial\Omega)$ and $H^{1/2}\Lambda^{k-1}(\partial\Omega)$, respectively.

The Hodge star $\star$ induces an isometric isomorphism, still denoted by $\star_{\partial\Omega}$,
\[
\star_{\partial\Omega}:N^{1/2}\Lambda^k(\partial\Omega)\to H^{1/2}\Lambda^{n-k-1}(\partial\Omega), \qquad \star_{\partial\Omega}\bigl(\ntr_{\partial\Omega}v\bigr):=\tr_{\partial\Omega}(\star v).
\]
This is well defined because $\ntr_{\partial\Omega}v=0$ implies $\tr_{\partial\Omega}(\star v)=0$, and it is isometric by the quotient-norm definitions and the fact that $\star:H^1\Lambda^{k+1}(\Omega)\to H^1\Lambda^{n-k-1}(\Omega)$ is an isometry.

\begin{proposition}[Minimal tangential trace lifting]\label{prop:trace.liftings}
For every $k\in\{0,\dots,n\}$, the tangential trace admits a linear right-inverse
\[
E_t^k:H^{1/2}\Lambda^k(\partial\Omega)\to H^1\Lambda^k(\Omega)
\]
with operator norm $1$.
\end{proposition}
\begin{proof}
This is a standard consequence of the quotient-norm definition of $H^{1/2}\Lambda^k(\partial\Omega)$, as every trace datum admits a unique minimal-norm representative in $(\ker\tr_{\partial\Omega})^\perp$, which defines a linear right-inverse of norm $1$.
\end{proof}

\begin{lemma}[Weak tangential and normal trace estimates]\label{lem:weak.trace.bounds}
For every $k\in\{0,\dots,n\}$, the following inequalities hold:
\[
\|\tr_{\partial\Omega}\omega\|_{T^{-1/2}\Lambda^k(\partial\Omega)} \le \sqrt2\,\|\omega\|_{H\Lambda^k(\Omega)} \qquad \forall \omega\in H\Lambda^k(\Omega),
\]
and
\[
\|\ntr_{\partial\Omega}v\|_{N^{-1/2}\Lambda^{k-1}(\partial\Omega)} \le \sqrt2\,\|v\|_{H^*\Lambda^k(\Omega)} \qquad \forall v\in H^*\Lambda^k(\Omega).
\]
\end{lemma}
\begin{proof}
We first prove the normal estimate. Let $\hat\varphi\in H^{1/2}\Lambda^{k-1}(\partial\Omega)$ and use the lifting $E_t^{k-1}\hat\varphi\in H^1\Lambda^{k-1}(\Omega)$ from Proposition~\ref{prop:trace.liftings}. By Green's formula \eqref{eq:weak.bdry.paring},
\[
\langle \ntr_{\partial\Omega}v,\hat\varphi\rangle_{\partial\Omega} = (\delta v,E_t^{k-1}\hat\varphi)_{L^2\Lambda^{k-1}(\Omega)}+(v,dE_t^{k-1}\hat\varphi)_{L^2\Lambda^k(\Omega)}.
\]
Hence
\[
|\langle \ntr_{\partial\Omega}v,\hat\varphi\rangle_{\partial\Omega}| \le\bigl(\|\delta v\|_{L^2\Lambda^{k-1}(\Omega)}+\|v\|_{L^2\Lambda^k(\Omega)}\bigr)\|E_t^{k-1}\hat\varphi\|_{H^1\Lambda^{k-1}(\Omega)}\le \sqrt2\,\|v\|_{H^*\Lambda^k(\Omega)}\|\hat\varphi\|_{H^{1/2}\Lambda^{k-1}(\partial\Omega)}.
\]
Taking the supremum over $\hat\varphi$ yields the bound for $\ntr_{\partial\Omega}v$.

The tangential estimate is trivial for $k=n$, since $\tr_{\partial\Omega}$ vanishes on $n$-forms. We therefore assume $k\le n-1$. Let $\hat\psi\in N^{1/2}\Lambda^k(\partial\Omega)$ and define
\[
E_n^k\hat\psi:=\pm\,\star\,E_t^{n-k-1}\bigl(\star_{\partial\Omega}\hat\psi\bigr),
\]
where the sign is chosen so that $\ntr_{\partial\Omega}(E_n^k\hat\psi)=\hat\psi$.
Since $\star$ and $\star_{\partial\Omega}$ are isometries and $\|E_t^{n-k-1}\|=1$, the map $E_n^k:N^{1/2}\Lambda^k(\partial\Omega)\to H^1\Lambda^{k+1}(\Omega)$ is a linear right-inverse of $\ntr_{\partial\Omega}$ with operator norm $1$. Therefore, using again \eqref{eq:weak.bdry.paring},
\[
\langle \tr_{\partial\Omega}\omega,\hat\psi\rangle_{\partial\Omega} = (d\omega,E_n^k\hat\psi)_{L^2\Lambda^{k+1}(\Omega)} +(\omega,\delta E_n^k\hat\psi)_{L^2\Lambda^k(\Omega)}.
\]
Hence
\[
|\langle \tr_{\partial\Omega}\omega,\hat\psi\rangle_{\partial\Omega}|
\le \bigl(\|d\omega\|_{L^2\Lambda^{k+1}(\Omega)}+\|\omega\|_{L^2\Lambda^k(\Omega)}\bigr)\|E_n^k\hat\psi\|_{H^1\Lambda^{k+1}(\Omega)}
\le \sqrt2\,\|\omega\|_{H\Lambda^k(\Omega)}\|\hat\psi\|_{N^{1/2}\Lambda^k(\partial\Omega)}.
\]
Taking the supremum over $\hat\psi$ yields the bound for $\tr_{\partial\Omega}\omega$.
\end{proof}

\begin{remark}[Generic trace constant]\label{rem:generic.trace.constant}
In the main text, we write $C_{\mathrm{tr}}(\Omega)$ for any constant dominating both weak trace bounds on $\Omega$. By Lemma~\ref{lem:weak.trace.bounds}, with the quotient-norm normalisation adopted above, one may take $C_{\mathrm{tr}}(\Omega)=\sqrt2$.
\end{remark}
%------------------------------------------------------------------------------%
% Bibliography
%------------------------------------------------------------------------------%

% \printbibliography
\addcontentsline{toc}{section}{Bibliography}
\bibliographystyle{abbrvnat}
\bibliography{sobolev-extension} 

\end{document}